\pdfoutput=1
\documentclass[10pt, leqno]{amsart}
\usepackage{amsfonts,amssymb}

\usepackage{amsmath}
\usepackage{amsthm}
\usepackage{mathtools}
\usepackage{amsrefs}
\usepackage[colorlinks=true,urlcolor=blue,linkcolor=blue,citecolor=blue]{hyperref}
\usepackage{doi}

\usepackage{qsymbols}
\usepackage{latexsym}
\usepackage{chngcntr}
\usepackage{xcolor}
\usepackage{paralist}
\usepackage{enumitem}
\usepackage{esint}
\usepackage{tikz-cd}

\newtheorem{theorem}{Theorem}[section]

\newtheorem{lemma}[theorem]{Lemma}
\newtheorem{proposition}[theorem]{Proposition}
\newtheorem{corollary}[theorem]{Corollary}
\theoremstyle{definition}
\newtheorem{definition}[theorem]{Definition}
\newtheorem{remark}[theorem]{Remark}

\newtheorem{lemdef}[theorem]{Lemma and Definition}

\newcommand{\IR}{\mathbb{R}}
\newcommand{\IC}{\mathbb{C}}
\newcommand{\IN}{\mathbb{N}}
\newcommand{\IZ}{\mathbb{Z}}

\newcommand{\IP}{\mathbb{P}}

\newcommand{\cK}{\mathcal{K}}

\newcommand{\cM}{\mathcal{M}}

\newcommand{\cF}{\mathcal{F}}
\newcommand{\cS}{\mathcal{S}}

\newcommand{\cR}{\mathcal{R}}
\newcommand{\cA}{\mathcal{A}}

\newcommand{\cC}{\mathcal{C}}
\newcommand{\cL}{\mathcal{L}}
\newcommand{\cD}{\mathcal{D}}
\newcommand{\cT}{\mathcal{T}}

\renewcommand{\L}{\mathrm{L}}
\newcommand{\C}{\mathrm{C}}
\newcommand{\B}{\mathrm{B}}
\renewcommand{\H}{\mathrm{H}}
\newcommand{\F}{\mathrm{F}}
	
\renewcommand{\S}{\mathrm{S}}
\newcommand{\W}{\mathrm{W}}

\newcommand{\fa}{\mathfrak{a}}

\newcommand{\e}{\mathrm{e}}
\newcommand{\ii}{\mathrm{i}}
\renewcommand{\d}{\mathrm{d}}
\newcommand{\eps}{\varepsilon}
\newcommand{\loc}{\mathrm{loc}}

\renewcommand\Im{\operatorname{Im}}
\newcommand\interior{\mathrm{in}}
\newcommand\exterior{\mathrm{ex}}
\newcommand{\Lop}{\mathcal{L}}
\newcommand{\pv}{\mathrm{p.v.}}

\newcommand{\divergence}{\operatorname{div}}

\DeclareMathOperator{\supp}{supp}
\DeclareMathOperator{\dist}{dist}

\DeclareMathOperator{\diam}{diam}

\DeclareMathOperator{\Id}{Id}
\DeclareMathOperator{\Rg}{\mathrm{ran}}

\DeclareMathOperator{\dom}{\mathcal{D}}

\numberwithin{equation}{section}

\title[The Stokes operator in two-dimensional bounded Lipschitz domains]{The Stokes operator\\ in two-dimensional bounded Lipschitz domains}
\author{Fabian Gabel}
\author{Patrick Tolksdorf}
\address{Institut f\"ur Mathematik, Technische Universit\"at Hamburg, Am Schwarzenberg-Campus 3, 21073 Hamburg, Germany}
\email{\href{mailto:fabian.gabel@tuhh.de}{fabian.gabel@tuhh.de}}
\address{Institut f\"ur Mathematik, Johannes Gutenberg-Universit\"at Mainz, Staudingerweg 9, 55099 Mainz, Germany}
\email{\href{mailto:tolksdorf@uni-mainz.de}{tolksdorf@uni-mainz.de}}

\thanks{}
\dedicatory{}
\hypersetup{
pdftitle={The Stokes operator in two-dimensional bounded Lipschitz domains},
pdfauthor={Fabian Gabel, Patrick Tolksdorf},
pdfcreator={LaTeX with hyperref (master, 3144c6e4)},
pdfkeywords={Dirichlet Problem, Lipschitz Domain, Stokes System, Layer Potential, Maximal Regularity, Navier--Stokes Equations},
pdfsubject={MSC2020 Primary 47D06, 35Q30; Secondary 76D03, 76D05, 76D07},
}

\begin{document}
\begin{abstract}
We consider the Stokes resolvent problem in a two-dimensional bounded Lipschitz domain $\Omega$ subject to homogeneous Dirichlet boundary conditions. We prove $\mathrm{L}^p$-resolvent estimates for $p$ satisfying the condition $\lvert 1 / p - 1 / 2 \rvert < 1 / 4 + \varepsilon$ for some $\varepsilon > 0$. We further show that the Stokes operator admits the property of maximal regularity and that its $\mathrm{H}^{\infty}$-calculus is bounded. This is then used to characterize domains of fractional powers of the Stokes operator. Finally, we give an application to the regularity theory of weak solutions to the Navier--Stokes equations in bounded planar Lipschitz domains.
\end{abstract}
\maketitle

\section{Introduction}

\noindent Consider the Stokes resolvent problem
\begin{align}
\label{Eq: Stokes resolvent}
 \left\{ \begin{aligned}
  \lambda u - \Delta u + \nabla \phi &= f && \text{in } \Omega, \\
  \divergence(u) &= 0 && \text{in } \Omega, \\
  u &= 0 && \text{on } \partial \Omega
 \end{aligned} \right.
\end{align}
in a bounded Lipschitz domain $\Omega \subset \IR^d$, $d \geq 2$. 
Here, the resolvent parameter $\lambda$ is supposed to lie in a sector $\S_{\theta} \coloneqq \{ z \in \IC \setminus \{ 0 \} : \lvert \arg(z) \rvert < \pi - \theta \}$, $\theta \in (0 , \pi / 2)$, in the complex plane. The Stokes resolvent problem was investigated in a plethora of articles dealing with all kinds of geometries of $\Omega$ and function spaces. See~\cites{Abels, Amann, Borchers_Sohr, Bolkart_Giga_Miura_Suzuki_Tsutsui, Farwig_Sohr, Geissert_Hess_Hieber_Schwarz_Stavrakidis, Giga, Kunstmann_Weis, Mitrea_Monniaux, Mitrea_Monniaux_Wright, Shen, Tolksdorf, Tolksdorf_Watanabe} to only mention a few.

A natural tool to consider the Stokes problem is the Helmholtz decomposition. Here, we say that the Helmholtz decomposition of $\L^p (\Omega ; \IC^d)$ exists if there is an algebraic and topological decomposition
\begin{align*}
 \L^p (\Omega ; \IC^d) = \L^p_{\sigma} (\Omega) \oplus \mathrm{G}_p (\Omega),
\end{align*}
where $\L^p_{\sigma} (\Omega)$ denotes the closure in $\L^p (\Omega ; \IC^d)$ of
\begin{align*}
 \C_{c , \sigma}^{\infty} (\Omega) \coloneqq \big\{ \varphi \in \C_c^{\infty} (\Omega ; \IC^d) : \divergence(\varphi) = 0 \big\}
\end{align*}
and where
\begin{align*}
 \mathrm{G}_p (\Omega) \coloneqq \big\{ g \in \L^p (\Omega ; \IC^d) : g = \nabla \Phi \text{ for some } \Phi \in \L^p_{\loc} (\Omega) \big\}.
\end{align*}
The projection onto the space $\L^p_{\sigma} (\Omega)$ is called the Helmholtz projection and is denoted by $\IP_p$. 
It belongs to the mathematical folklore that, in the case $p = 2$, this projection is an orthogonal projection, and $\mathrm{G}_2 (\Omega)$ is the orthogonal complement to $\L^2_{\sigma} (\Omega)$ with respect to the $\L^2$-inner product. 
In particular, the Helmholtz decomposition of $\L^2(\Omega; \IC^d)$ exists for each open set $\Omega$. 
However, the situation for $p \neq 2$ is different and much more sensitive to the underlying geometry. Even smoothness of $\partial \Omega$ does not imply for certain unbounded domains the existence of the Helmholtz decomposition as the classic example of Maslennikova and Bogovski\u{\i}~\cite{Maslennikova_Bogovskii} shows. Also, irregularity of the boundary destroys the existence of the Helmholtz decomposition in certain $\L^p$-spaces. This shows in the works of Fabes, Mendez, and Mitrea~\cite{Fabes_Mendez_Mitrea} and of Mitrea~\cite{MitreaD}. 
Indeed, in~\cite{Fabes_Mendez_Mitrea} it is proved that, for each bounded Lipschitz domain $\Omega \subset \IR^d$, $d \geq 3$, there exists a constant $\eps > 0$ such that the Helmholtz decomposition exists on $\L^p (\Omega ; \IC^d)$ whenever
\begin{align}
\label{Eq: Helmholtz large dimension}
\Big\lvert \frac{1}{p} - \frac{1}{2} \Big\rvert < \frac{1}{6} + \eps.
\end{align}
In particular, it is shown that this range of numbers $p$ is sharp, i.e., for each $1 < p < \infty$ that does not lie in the interval $[3 / 2 , 3]$, a bounded Lipschitz domain is constructed such that the Helmholtz decomposition of $\L^p (\Omega ; \IC^d)$ fails. An analogous result was proved in~\cite[Thm.~1.1]{MitreaD} for bounded planar Lipschitz domains. Here, the Helmholtz decomposition exists on $\L^p (\Omega ; \IC^2)$ whenever
\begin{align}
\label{Eq: Helmholtz two dimensions}
 \Big\lvert \frac{1}{p} - \frac{1}{2} \Big\rvert < \frac{1}{4} + \eps.
\end{align}

\indent There are strong indications that the existence of the Helmholtz decomposition leads to a rich functional analytic theory of the Stokes operator. Indeed, for uniform $\C^3$-domains (bounded and unbounded), this is the message of the article of Geissert, Heck, Hieber, and Sawada~\cite{Geissert_Heck_Hieber_Sawada}. However, let us mention that the existence of the Helmholtz decomposition is not necessary for a rich functional analytic theory of the Stokes operator as the article~\cite{Bolkart_Giga_Miura_Suzuki_Tsutsui} shows. 

The results in~\cite{Fabes_Mendez_Mitrea} led Taylor to conjecture in~\cite{Taylor} that, for each three-dimensional bounded Lipschitz domain, there exists $\eps > 0$ such that, for each $p$ in the range~\eqref{Eq: Helmholtz large dimension}, the Stokes operator gives rise to a bounded analytic semigroup. In 2001, one year after Taylor formulated this conjecture, Deuring~\cite{Deuring} constructed bounded Lipschitz domains such that, for $p$ large enough, the Stokes operator indeed fails to generate even a strongly continuous semigroup on $\L^p_{\sigma} (\Omega)$. Compare also the discussion of Deuring's result in the article of Monniaux and Shen~\cite[Sect.~6]{Monniaux_Shen}.

Mitrea and Monniaux were able to prove Taylor's conjecture for the Stokes operator but with \emph{Neumann type boundary conditions} in 2009, see \cite{Mitrea_Monniaux_Neumann}.
In 2012, Taylor's original conjecture was settled in the affirmative by Shen~\cite{Shen} and later extended by Dikland to the case of weighted $\L^p$-spaces~\cite{Dikland}. 
Analogously to Taylor, Mitrea~\cite{MitreaD} conjectured that the Stokes operator gives rise to a bounded analytic semigroup on $\L^p_{\sigma} (\Omega)$ in \emph{two-dimensional} bounded Lipschitz domains $\Omega$ if $p$ satisfies~\eqref{Eq: Helmholtz two dimensions}. 
In the first part of this work, we will briefly review Shen's proof and give the necessary adaptations to the two-dimensional case. 
Thereby, we will give an affirmative solution of the conjecture stated in~\cite[Conj.~1.2]{MitreaD} and prove the following theorem.

\begin{theorem}
\label{Thm: Resolvent}
Let $\Omega \subset \IR^2$ be a bounded Lipschitz domain and $\theta \in (0 , \pi / 2)$. Then there exists $\eps > 0$ such that, for all numbers $p$ that satisfy~\eqref{Eq: Helmholtz two dimensions}, there exists a constant $C > 0$ such that, for all $f \in \L^p_{\sigma} (\Omega) \cap \L^2_{\sigma} (\Omega)$ and $\lambda \in \S_{\theta}$, the unique weak solutions $u \in \W^{1,2}_0 (\Omega ; \IC^2)$ and $\phi \in \L^2 (\Omega)$ with $\int_{\Omega} \phi \; \d x = 0$ to~\eqref{Eq: Stokes resolvent} satisfy
\begin{align*}
 (1 + \lvert \lambda \rvert) \| u \|_{\L^p_{\sigma} (\Omega)} \leq C \| f \|_{\L^p_{\sigma} (\Omega)}.
\end{align*}
Here, $\eps$ depends on $\theta$, the Lipschitz character of $\Omega$, and $\diam(\Omega)$ and $C$ depends on $\theta$, $p$, the Lipschitz character of $\Omega$, and $\diam(\Omega)$.
\end{theorem}

To put this into a functional analytic context, let us introduce the Stokes operator as the realization of the sesquilinear form
\begin{align*}
 \fa \colon \W^{1 , 2}_{0 , \sigma} (\Omega) \times \W^{1 , 2}_{0 , \sigma} (\Omega) \to \IC, \quad (u , v) \mapsto \int_{\Omega} \nabla u \cdot \overline{\nabla v} \; \d x \coloneqq \sum_{\alpha , \beta = 1}^2 \int_{\Omega} \partial_{\alpha} u_{\beta} \overline{\partial_{\alpha} v_{\beta}} \; \d x.
\end{align*}
Here, we used the notation
\begin{align*}
 \W^{1 , p}_{0 , \sigma} (\Omega) \coloneqq \overline{\C_{c , \sigma}^{\infty} (\Omega)}^{\|\cdot\|_{\W^{1,p}(\Omega; \IC^2)}} , \qquad 1 < p < \infty.
\end{align*}
Now, the domain of the Stokes operator $A_2$ on $\L^2_{\sigma} (\Omega)$ is defined as
\begin{align*}
 \dom(A_2) \coloneqq \Big\{ u \in \W^{1 , 2}_{0 , \sigma} (\Omega) : \, \exists f \in \L^2_{\sigma} (\Omega) \text{ s.t.\@ for all } v \in \W^{1 , 2}_{0 , \sigma} (\Omega) \text{ it holds } \fa(u , v) = \int_{\Omega} f \cdot \overline{v} \; \d x \Big\}
\end{align*}
and, for $u \in \dom(A_2)$ with corresponding function $f \in \L^2_{\sigma} (\Omega)$, we define
\begin{align*}
 A_2 u \coloneqq f.
\end{align*}
For $p > 2$ satisfying~\eqref{Eq: Helmholtz two dimensions}, the realization of $A_2$ on $\L^p_{\sigma} (\Omega)$ is denoted by $A_p$ and given as the part of $A_2$ in $\L^p_{\sigma} (\Omega)$ which is defined by
\begin{align*}
 \dom(A_p) \coloneqq \big\{ u \in \dom(A_2) \cap \L^p_{\sigma} (\Omega) : A_2 u \in \L^p_{\sigma} (\Omega) \big\}, \quad A_p u \coloneqq A_2 u \quad \text{for} \quad u \in \dom(A_p).
\end{align*}
For $p < 2$ satisfying~\eqref{Eq: Helmholtz two dimensions} and $1 / p + 1 / p^{\prime} = 1$, define $A_p$ to be the $\L^p_{\sigma}$-adjoint of the operator $A_{p^{\prime}}$. 
With these definitions of the Stokes operator on $\L^p_{\sigma} (\Omega)$, Theorem~\ref{Thm: Resolvent} ensures that $A_p$ is an invertible and sectorial operator. In particular, we have the following corollary to Theorem~\ref{Thm: Resolvent}.

\begin{corollary}
\label{Cor: Analyticity}
Let $\Omega \subset \IR^2$ be a bounded Lipschitz domain. Then there exists $\eps > 0$ such that, for all numbers $p$ that satisfy~\eqref{Eq: Helmholtz two dimensions}, the operator $- A_p$ generates an exponentially stable analytic semigroup on $\L^p_\sigma(\Omega)$. The constant $\eps$ depends only on the Lipschitz character of $\Omega$ and $\diam(\Omega)$.
\end{corollary}

Following Haase~\cite[Chap.~3]{Haase}, we define fractional powers of the Stokes operators $A_p$. In the case $p = 2$, the domains of the fractional powers $\dom(A_2^{\alpha})$  for $0 \leq \alpha < 3 / 4$ were characterized in terms of suitable Bessel potential spaces $\H^{\alpha / 2 , 2}_{0 , \sigma} (\Omega)$ by Mitrea and Monniaux~\cite[Thm.~5.1]{Mitrea_Monniaux}. The following theorem gives an $\L^p_{\sigma}$-version of this result. This generalizes the results of Giga~\cite[Thm.~3]{Giga_fractional} from the smooth situation to the situation of planar and bounded Lipschitz domains.

\begin{theorem}\label{Thm: Fractional powers}
Let $\Omega \subset \IR^2$ be a bounded Lipschitz domain. 
Then there exist $\eps  > 0$ and $\delta \in (0,1]$ such that, for all $1 < p < \infty$ and $0< \theta < 1$ satisfying~\eqref{Eq: Helmholtz two dimensions} and either
\begin{align*}
 \theta < \frac{1}{2} + \frac{1}{2 p} \quad \text{if} \quad \frac{1}{2} - \frac{1}{p} \leq \frac{\delta}{2}
 \qquad
 \text{or}
 \qquad
 \theta < \frac{1}{p} + \frac{1 + \delta}{4} \quad \text{if} \quad \frac{1}{2} - \frac{1}{p} > \frac{\delta}{2}\;,
\end{align*}
we have with equivalent norms that
\begin{align*}
 \dom(A_p^{\theta}) = \H^{2 \theta , p}_{0 , \sigma} (\Omega).
\end{align*}
\end{theorem}

As a corollary, we obtain the following $\L^p$--$\L^q$-smoothing estimates for the Stokes semigroup.

\begin{corollary}
Let $\Omega \subset \IR^2$ be a bounded Lipschitz domain. 
Then there exists $\eps > 0$ such that, for all $1 < p \leq q < \infty$ satisfying~\eqref{Eq: Helmholtz two dimensions}, there exists $C > 0$ such that
\begin{align*}
 \| \e^{- t A_p} f \|_{\L^q_{\sigma} (\Omega)} &\leq C t^{- (\frac{1}{p} - \frac{1}{q})} \| f \|_{\L^p_{\sigma} (\Omega)}, \qquad t > 0 ,\, f \in \L^p_{\sigma} (\Omega),
\intertext{and}
\| \nabla \e^{- t A_p} f \|_{\L^q (\Omega; \IC^{2\times 2})} &\leq C t^{- \frac{1}{2} - (\frac{1}{p} - \frac{1}{q})} \| f \|_{\L^p_{\sigma} (\Omega)}, \qquad t > 0 ,\, f \in \L^p_{\sigma} (\Omega).
\end{align*}
\end{corollary}

In order to prove Theorem~\ref{Thm: Fractional powers}, we show that the $\H^{\infty}$-calculus of $A_p$ is bounded. 
Here, we say that an injective operator $A$ on a Banach space $X$ that is sectorial of some angle $\omega \in [0 , \pi)$ has a bounded $\H^{\infty}$-calculus if, for some  $\theta \in (0 , \pi - \omega)$, there exists $C > 0$ such that, for all bounded analytic functions $f$ on $\S_\theta$, the estimate 
\begin{align*}
    \|f(A) \|_{\cL(X)} \leq C \sup_{z \in \S_\theta} |f(z)| 
\end{align*}
holds. 
Here, the expression $f(A)$ is to be understood in the sense of a \emph{regularization} of the \emph{natural functional calculus} for sectorial operators, cf.~\cite[Sect.~2.3]{Haase}.

\begin{theorem}
\label{Thm: Hinfty}
Let $\Omega \subset \IR^2$ be a bounded Lipschitz domain. Then there exists $\eps > 0$ such that, for all numbers $p$ that satisfy~\eqref{Eq: Helmholtz two dimensions}, the $\H^{\infty}$-calculus of $A_p$ is bounded.
\end{theorem}

The boundedness of the $\H^{\infty}$-calculus will be deduced by using a comparison principle due to Kunstmann and Weis~\cite{Kunstmann_Weis} in which the Stokes operator is compared to the Dirichlet-Laplacian. A crucial ingredient for this comparison principle is the $\cR$-sectoriality of the Stokes operator that is well-known to be equivalent to the property of maximal $\L^q$-regularity, see Section~\ref{Sec: Resolvent estimates and maximal regularity} for further information, and is established in the following theorem.

\begin{theorem}
\label{Thm: Maximal regularity}
Let $\Omega \subset \IR^2$ be a bounded Lipschitz domain. Then there exists $\eps > 0$ such that, for all numbers $p$ that satisfy~\eqref{Eq: Helmholtz two dimensions} and for all $1 < q < \infty$, the Stokes operator $A_p$ has maximal $\L^q$-regularity.
\end{theorem}

Additionally to this theorem, we employ the \emph{square root property} $\cD(A_p^{1 / 2}) = \W^{1 , p}_{0 , \sigma} (\Omega)$ that is established in Theorem~\ref{Thm: Fractional powers} to transfer the maximal regularity property from the ground space $X = \L^p_{\sigma} (\Omega)$ to the ground space $X = \W^{-1 , p}_{\sigma} (\Omega) := [\W^{1 , p^{\prime}}_{0 , \sigma} (\Omega)]^*$ and establish the maximal $\L^q$-regularity property for the \emph{weak} Stokes operator defined on $\W^{-1 , p}_{\sigma} (\Omega)$, see Section~\ref{sec: weak Stokes}. 

Finally, these properties will be used to investigate regularity properties of Leray--Hopf weak solutions to the Navier--Stokes equations in a bounded Lipschitz domain $\Omega$
\begin{align}\label{Eq: NSE}
 \left\{ \begin{aligned}
 u^{\prime} - \Delta u + (u \cdot \nabla) u + \nabla \pi &= f = f_0 + \IP_2 \divergence (F) && \text{in } (0 , \infty) \times \Omega, \\
 \divergence (u) &= 0 && \text{in } (0 , \infty) \times \Omega, \\
 u &= 0 && \text{on } (0 , \infty) \times \partial \Omega, \\
 u (0) &= u_0 && \text{in } \Omega.
 \end{aligned} \right.
\end{align}
Weak solutions in the \emph{Leray--Hopf class}
\begin{align*}
 \L\H_{\infty} := \L^{\infty} (0 , \infty ; \L^2_{\sigma} (\Omega)) \cap \L^2 (0 , \infty ; \W^{1 , 2}_{0 , \sigma} (\Omega))
\end{align*}
are known to exist since the seminal works of Leray~\cite{Leray} and Hopf~\cite{Hopf}. In particular, in the two-dimensional case and if $\Omega$ is smooth enough, e.g., if the boundary is $\C^{1 , 1}$-regular, then Leray--Hopf weak solutions $u$ to~\eqref{Eq: NSE}, with $F = 0$ and $u_0$ and $f_0$ regular enough, are known to be unique and regular, i.e., that 
\begin{align*}
 u \in \L^{\infty} (0 , \infty ; \W^{1 , 2}_{0 , \sigma} (\Omega)) \cap \L^2 (0 , \infty ; \W^{2 , 2} (\Omega ; \IC^2)) \cap \W^{1 , 2} (0 , \infty ; \L^2_{\sigma} (\Omega)).
\end{align*}
If $\Omega$ is merely Lipschitz regular, then such regularity properties break down in general. Our final theorem establishes regularity properties of Leray--Hopf weak solutions in this geometric setting.

\begin{theorem}\label{Thm: Navier-Stokes regularity}
Let $\Omega \subset \IR^2$ be a bounded Lipschitz domain. 
Then there exists $\eps > 0$, depending only on the Lipschitz geometry of $\Omega$, such that the following statements are valid.
\begin{enumerate}[label=\normalfont{(\roman*)}]
 \item \label{Item: Lp theory} For all numbers $1 < s < 2$ and $1 < p < 2$ that satisfy
 \begin{align*}
  1 - \frac{1}{s} = \frac{1}{p} - \frac{1}{2} < \frac{1}{4} + \eps
 \end{align*}
 and all Leray--Hopf weak solutions
 \begin{align*}
  u \in \L^{\infty} (0 , \infty ; \L^2_{\sigma} (\Omega)) \cap \L^2 (0 , \infty ; \W^{1 , 2}_{0 , \sigma} (\Omega))
 \end{align*}
 to~\eqref{Eq: NSE} with initial data $u_0$ and force $f = f_0$ satisfying
 \begin{align*}
  u_0 \in \big( \L^p_{\sigma} (\Omega) , \dom(A_p) \big)_{1 - \frac{1}{s} , s} \quad \text{and} \quad f_0 \in \L^s (0 , \infty ; \L^p_{\sigma} (\Omega)),
 \end{align*}
 one has that
 \begin{align*}
 u \in \W^{1 , s} (0 , \infty ; \L^p_{\sigma} (\Omega)) \cap \L^s (0 , \infty ; \dom(A_p)).
\end{align*}
 \item\label{Item: Distribution theory} For all numbers $1 < p < \infty$ that satisfy~\eqref{Eq: Helmholtz two dimensions}, all $1 < s < \infty$ that satisfy
 \begin{align*}
 \frac{1}{p} + \frac{1}{s} = 1 ,
 \end{align*}
 and all Leray--Hopf solutions
 \begin{align*}
  u \in \L^{\infty} (0 , \infty ; \L^2_{\sigma} (\Omega)) \cap \L^2 (0 , \infty ; \W^{1 , 2}_{0 , \sigma} (\Omega))
 \end{align*}
 to~\eqref{Eq: NSE} with initial data $u_0$ and force $f = \IP_2 \divergence(F)$ satisfying
 \begin{align*}
  u_0 \in \big( \W^{-1 , p}_{\sigma} (\Omega) , \W^{1 , p}_{0 , \sigma} (\Omega) \big)_{1 - \frac{1}{s} , s} \quad \text{and} \quad F \in \L^s (0 , \infty ; \L^p (\Omega ; \IC^{2 \times 2})),
 \end{align*}
 one has that
 \begin{align*}
 u \in \W^{1 , s} (0 , \infty ; \W^{-1 , p}_{\sigma} (\Omega)) \cap \L^s (0 , \infty ; \W^{1 , p}_{0 , \sigma} (\Omega)).
\end{align*}
 \end{enumerate}
\end{theorem}

\section{An overview of Shen's proof with adaptations to two dimensions}

\noindent Shen's proof of the resolvent estimates of the Stokes operator in $d \geq 3$ dimensions fundamentally bases on the resolution of the $\L^2$-Dirichlet problem for the Stokes resolvent system \cite[Thm.~1.1]{Shen}. 
To formulate this problem, we state the following lemma and definition, cf.~\cite[p.~208]{Shen_lectures}.

\begin{lemdef}
\label{Lem: Cones}
Let $\Xi \subset \IR^d$, $d \geq 2$, be a bounded Lipschitz domain. There exists $\alpha > 1$ depending only on $d$ and the Lipschitz character of $\Xi$ such that each of the sets
\begin{align*}
 \gamma_{\alpha}^\interior (q) &\coloneqq \big\{ x \in \Xi : \lvert x - q \rvert < \alpha \dist(x , \partial \Xi) \big\}, \qquad q \in \partial \Xi, \\
 \gamma_{\alpha}^\exterior (q) &\coloneqq \big\{ x \in \overline{\Xi}^c: \lvert x - q \rvert < \alpha \dist(x , \partial \Xi) \big\}, \qquad q \in \partial \Xi,
\end{align*}
contains a cone of fixed height and aperture with vertex at $q$. In this case, we call the family $\{ \gamma_{\alpha}^\interior (q) : q \in \partial \Xi \}$ an \textit{interior} and $\{ \gamma_{\alpha}^\exterior (q) : q \in \partial \Xi \}$ an \textit{exterior} \textit{family of non-tangential approach regions}.
\end{lemdef}

In the following, we fix a value of $\alpha > 1$ subject to Lemma and Definition~\ref{Lem: Cones}. The notions of the non-tangential maximal function and non-tangential convergence are introduced as follows.

\begin{definition}
\label{Def: Nontangential maximal functions}
Let $\Xi \subset \IR^d$, $d \geq 2$, be a bounded Lipschitz domain and $\{ \gamma_{\alpha}^\interior (q) : q \in \partial \Xi \}$ a corresponding family of non-tangential approach regions. 
For a function $u \colon \Xi \to \IC^m$, $m \in \IN$, the \textit{interior non-tangential maximal function of $u$} is defined as
\begin{align*}
  (u)^*_\interior (q) \coloneqq \sup \big\{ \lvert u (x) \rvert : x \in \gamma_{\alpha}^\interior(q) \big\},
  \qquad q \in \partial \Xi.
\end{align*}
Similarly, for a function $v \colon \overline{\Xi}^c \to \IC^m$ and an exterior family of non-tangential approach regions $\{ \gamma_{\alpha}^\exterior (q) : q \in \partial \Xi \}$, the \textit{exterior non-tangential maximal function of $v$} is defined analogously and denoted by~$(v)^*_\exterior$.

For a function $f \colon \partial \Xi \to \IC^m$, we say that $u = f$ \textit{in the sense of non-tangential convergence from the inside} if
\begin{align*}
 \lim_{\substack{x \in \gamma_{\alpha}^\interior (q) \\ x \to q}} u(x) = f (q), \qquad \text{a.e.\@ } q \in \partial \Xi,
\end{align*}
and we call $f$ the \emph{non-tangential limit of $u$ inside of $\Xi$}.
Analogously, we say that $v = f$ \textit{in the sense of non-tangential convergence from the outside} if
\begin{align*}
 \lim_{\substack{x \in \gamma_{\alpha}^\exterior (q) \\ x \to q}} v(x) = f (q), \qquad \text{a.e.\@ } q \in \partial \Xi,
\end{align*}
and we call $f$ the \emph{non-tangential limit of $v$ outside of $\Xi$}.
\end{definition}

Let $\theta \in (0 , \pi / 2)$, and let
\begin{align*}
 \S_{\theta} \coloneqq \big\{ z \in \IC \setminus \{ 0 \} : \lvert \arg(z) \rvert < \pi - \theta \big\}.
\end{align*}
For a bounded Lipschitz domain $\Xi \subset \IR^d$, $d \geq 2$, with connected boundary and $\lambda \in \S_{\theta}$, consider the Dirichlet problem
\begin{align} \tag{Dir} \label{Dir}
 \left\{ \begin{aligned}
  \lambda u - \Delta u + \nabla \phi &= 0 && \text{in } \Xi ,\\
  \divergence(u) &= 0 && \text{in } \Xi, \\
  u &= g && \text{on } \partial \Xi, \\
  (u)_\interior^* &\in \L^2 (\partial \Xi),
 \end{aligned} \right.
\end{align}
where the equality $u = g$ on $\partial \Xi$ is to be understood in the sense of non-tangential convergence from the inside. Due to the condition $\divergence(u) = 0$, the divergence theorem implies that the normal component of $g$, i.e., $g \cdot n$, has average zero on $\partial \Xi$. Thus, the definition of the boundary space
\begin{align*}
 \L^2_n (\partial \Xi) \coloneqq \Big\{ g \in \L^2 (\partial \Xi ; \IC^d) : \int_{\partial \Xi} g \cdot n \, \d \sigma = 0 \Big\}
\end{align*}
seems natural.
The key ingredient that allowed Shen to prove the counterpart to Theorem~\ref{Thm: Resolvent} in three and more dimensions is the following resolution of the $\L^2$-Dirichlet problem:

\begin{theorem}[$\L^2$-Dirichlet problem]
\label{Thm: Dirichlet}
Let $\Xi \subset \IR^d$, $d \geq 2$, be a bounded Lipschitz domain with connected boundary and let $\theta \in (0 , \pi / 2)$. To every resolvent parameter $\lambda \in \S_{\theta}$ and every function $g \in \L^2_n (\partial \Xi)$, there exists a unique smooth function $u \colon \Xi \to \IC^d$ that satisfies $(u)_\interior^* \in \L^2 (\partial \Xi)$ and a smooth function $\phi \colon \Xi \to \IC$ that is unique up to the addition of constants such that~\eqref{Dir} is satisfied. Moreover, there exists a constant $C > 0$ such that
\begin{align}
\label{Eq: Stability estimate}
 \| (u)_\interior^* \|_{\L^2 (\partial \Xi)} \leq C \| g \|_{\L^2 (\partial \Xi)}.
\end{align}
The constant $C$ depends only on $d$, $\theta$, the Lipschitz character of $\Xi$, and on constants $\alpha , \beta > 0$ that satisfy $\alpha \leq \diam(\Xi) \leq \beta$.
\end{theorem}

\begin{remark}
As we will see in the proof, the functions $u$ and $\phi$ have representations as double layer potentials.
\end{remark}

The case $d \geq 3$ of Theorem~\ref{Thm: Dirichlet} was already proved by Shen~\cite[Thm.~5.5]{Shen}. In order to make the main idea of the proof accessible for the case $d = 2$, we derive estimates on fundamental solutions of the Stokes resolvent problem in Subsection~\ref{Sec: Properties of the matrix of fundamental solutions to the Stokes resolvent problem}.
Afterward, in Subsection~\ref{Sec: Single and double layer potentials}, we sketch the proof of Theorem~\ref{Thm: Dirichlet} and add remarks on the two-dimensional case.

\subsection{Properties of the matrix of fundamental solutions to the Stokes resolvent problem}
\label{Sec: Properties of the matrix of fundamental solutions to the Stokes resolvent problem}

This section aims to study fundamental solutions to the Stokes resolvent problem. Before working on the Stokes resolvent problem, we will look at the atoms of the fundamental solution of this problem: the Hankel functions. 

First of all, let us fix some recurring quantities. Let $\theta \in (0 , \pi / 2)$ and let $\lambda \in \S_{\theta}$. The polar form of $\lambda$ is given by $\lambda = r \e^{\ii \tau}$ with $0 < r < \infty$ and $- \pi + \theta < \tau < \pi - \theta$. 
Now, set
\begin{align*}
 k \coloneqq \sqrt{r}\, \e^{\ii (\pi + \tau) / 2}.
\end{align*}
This definition ensures that $k^2 = - \lambda$ and that $\theta / 2 < \arg(k) < \pi - \theta / 2$. The definition of $k$ further ensures the estimate
\begin{align}
\label{Eq: Imaginary part of square root}
 \Im(k) > \sqrt{\lvert \lambda \rvert} \sin(\theta / 2) > 0.
\end{align}
Before diving into fundamental solutions of the Stokes resolvent problem, we will first consider a fundamental solution for the \emph{scalar Helmholtz equation} in $\IR^2$
\begin{equation}\label{Eq: Scalar Helmholtz}
 \lambda u - \Delta u = 0.
\end{equation}
One fundamental solution to~\eqref{Eq: Scalar Helmholtz} with pole at the origin is given by
\begin{align*}
 G(x ; \lambda) \coloneqq \frac{\ii}{4} H_0^{(1)} (k \lvert x \rvert),
\end{align*}
where $H^{(1)}_{\nu} (z)$ is the \emph{Hankel function of the first kind}.
For $\nu > - 1 / 2$ and $0 < \arg(z) < \pi$, Hankel functions of the first kind have an integral representation
\begin{align}
\label{Eq: Hankel}
 H^{(1)}_{\nu} (z) 
 = \frac{2^{\nu + 1} \e^{\ii (z - \nu \pi)} z^{\nu}}{\ii \sqrt{\pi}\, \Gamma(\nu + \frac{1}{2})} \int_0^{\infty} \e^{2 z \ii s} s^{\nu - \frac{1}{2}} (1 + s)^{\nu - \frac{1}{2}} \, \d s,
\end{align}
which can be found in Lebedev's book~\cite[Sect.~5.11]{Lebedev}.
In the following, we will take $z = k \lvert x \rvert$. Note that the condition on the argument of $z$ is fulfilled due to~\eqref{Eq: Imaginary part of square root}.

\begin{lemma}
\label{Lem: Estimates on Helmholtz fundamental solution}
Let $\lambda \in \S_{\theta}$ and $\ell \geq 1$. 
There exist constants $c > 0$ depending only on $\theta$ and $C_{\ell} > 0$ depending only on $\ell$ and $\theta$ such that
\begin{align*}
 \lvert \nabla_x^{\ell} G(x ; \lambda) \rvert \leq \frac{C_{\ell} \e^{- c \sqrt{\lvert \lambda \rvert} \lvert x \rvert}}{\lvert x \rvert^{\ell}}, \qquad x \in \IR^2 \setminus \{ 0 \}.
\end{align*}
In the case $\ell = 0$, this inequality holds for all $x \in \IR^2$ with $\lvert \lambda \rvert \lvert x \rvert^2 \geq 1$.
\end{lemma}

\begin{proof}
Let $\ell \geq 1$. To estimate $\nabla^{\ell} G(x ; \lambda)$, notice the validity of the identity
\begin{align}
\label{Eq: Iterative derivative}
 \frac{\d}{\d z} \big\{ z^{- \nu} H_{\nu}^{(1)} (z) \big\} = - z^{- \nu} H_{\nu + 1}^{(1)} (z),
\end{align}
which holds for $z \in \IC \setminus (- \infty , 0]$ and $\nu \in \IR$, see~\cite[Eq.~(5.6.3)]{Lebedev}. The integral representation~\eqref{Eq: Hankel} further gives for $\nu \geq 1 / 2$ and $\Im(z) > 0$
\begin{align}
\label{Eq: Pointwise integral estimate}
 \lvert H^{(1)}_{\nu} (z) \rvert \leq C \e^{- \Im(z)} \lvert z \rvert^{\nu} \int_0^{\infty} \e^{- 2 s \Im(z)} s^{\nu - \frac{1}{2}} (1 + s)^{\nu - \frac{1}{2}} \, \d s \leq C \lvert z \rvert^{\nu} \lvert \Im(z) \rvert^{- 2 \nu} \e^{- \Im(z) / 2}.
\end{align}
Now, the desired estimate follows by combining~\eqref{Eq: Iterative derivative} with~\eqref{Eq: Pointwise integral estimate} via the chain rule.

In order to derive the estimate in the case $\ell = 0$ and $x$ subject to $\lvert \lambda \rvert \lvert x \rvert^2 \geq 1$, use that here we have $\Im (z) > \sqrt{\lvert \lambda \rvert} \lvert x \rvert \sin(\theta / 2) \geq \sin(\theta / 2)$ and conclude by virtue of~\eqref{Eq: Hankel} that
\begin{align*}
 \lvert H^{(1)}_0 (z) \rvert \leq C \e^{- \Im(z)} \int_0^{\infty} \e^{- 2 s \sin(\theta / 2)} s^{- \frac{1}{2}} (1 + s)^{- \frac{1}{2}} \, \d s. & \qedhere
\end{align*}
\end{proof}

We will use the following interior estimate for solutions to Poisson's equation on balls $B(x , r)$ centered in $x \in \IR^2$ with radius $r > 0$ during the derivation of the subsequent estimates in this section. 
In the case $\ell = 1$, the interior estimate is a special case of~\cite[Eq.~(3.16)]{Gilbarg_Trudinger}, and in the case $\ell > 1$, it follows by induction. We omit further details.

\begin{lemma}
\label{Lem: Interior estimates Poisson}
Let $r > 0$ and $x \in \IR^2$. If $w \in \C^k (B(x , r)) \cap \C^0 (\overline{B(x , r)})$ is a solution to $\Delta w = f$ in $B(x , r)$ for $f \in \C^{k - 1} (B(x , r))$, then
\begin{align*}
 \lvert \nabla^{\ell} w (x) \rvert \leq C r^{- \ell} \sup_{B(x , r)} \lvert w \rvert + C \max_{0 \leq j \leq \ell - 1} \sup_{B(x , r)} r^{j - \ell + 2} \lvert \nabla^j f \rvert, \quad \ell \leq k,
\end{align*}
where $C > 0$ only depends on $\ell$.
\end{lemma}

In addition to the fundamental solution $G(x; \lambda)$ to the scalar Helmholtz equation~\eqref{Eq: Scalar Helmholtz}, let 
\begin{align*}
 G(x ; 0) \coloneqq - \frac{1}{2 \pi} \log (\lvert x \rvert), \qquad x \in \IR^2 \setminus \{ 0 \},
\end{align*} 
denote the fundamental solution to the Laplace equation in the plane. 
The following lemma gives estimates on derivatives of $G(x ; \lambda) - G(x ; 0)$.

\begin{lemma}\label{Lem: Difference Helmholtz}
Let $\lambda \in \S_{\theta}$. Then, for each $\ell \geq 3$, there exists $C > 0$ depending only on $\ell$ and $\theta$ such that
\begin{align}
\label{Eq: Difference Helmholtz}
 \big\lvert \nabla^{\ell}_x \big\{ G(x ; \lambda) - G(x ; 0) \big\} \big\rvert \leq C \lvert \lambda \rvert \lvert x \rvert^{2 - \ell}.
\end{align}
\end{lemma}

\begin{proof}
\textbf{Step 1: the case $\lvert \lambda \rvert \lvert x \rvert^2 > 1 / 2$.} 
In this case, Lemma~\ref{Lem: Estimates on Helmholtz fundamental solution} readily yields
\begin{align*}
 \big\lvert \nabla_x^{\ell} \big\{ G(x ; \lambda) - G(x ; 0) \big\} \big\rvert \leq C \bigg\{ \frac{\e^{- c \sqrt{\lvert \lambda \rvert} \lvert x \rvert}}{\lvert x \rvert^{\ell}} + \frac{1}{\lvert x \rvert^{\ell}} \bigg\} \leq C \frac{\lvert \lambda \rvert}{\lvert x \rvert^{\ell - 2}}\,.
\end{align*}
\medbreak
\noindent \textbf{Step 2: reduction to the case $\ell = 3$ and $\lvert \lambda \rvert \lvert x \rvert^2 \leq 1 / 2$.} Having the first part of the proof at our disposal, we concentrate on the case $\lvert \lambda \rvert \lvert x \rvert^2 \leq 1 / 2$. We show that~\eqref{Eq: Difference Helmholtz} can be established in the case $\ell \geq 4$ under the validity of~\eqref{Eq: Difference Helmholtz} in the case $\ell = 3$. Indeed, if we use the definition $w (x) \coloneqq \nabla_x^3 \{ G(x ; \lambda) - G(x ; 0)\}$, we get that $\Delta_x w (x) = \lambda \nabla_x^3 G(x ; \lambda)$ in $\IR^2 \setminus \{ 0 \}$. Employing Lemma~\ref{Lem: Interior estimates Poisson} with $f(x) \coloneqq \lambda \nabla_x^3 G(x ; \lambda)$ and $0 < r < \lvert x \rvert$ yields
\begin{alignat*}{2}
 \big\lvert \nabla^{\ell - 3} w(x) \big\rvert &\leq C r^{3 - \ell} \sup_{B(x , r)} \lvert w \rvert 
 &&+ C \max_{0 \leq j \leq \ell - 4} \sup_{B(x , r)} r^{j - \ell + 5} \lvert \nabla^j f \rvert \\
 &\leq C r^{3 - \ell} \sup_{y \in B(x , r)} \lvert \lambda \rvert \lvert y \rvert^{-1} 
 &&+ C \max_{0 \leq j \leq \ell - 4} \sup_{y \in B(x , r)} r^{j - \ell + 5} \lvert \lambda \rvert \lvert y \rvert^{- 3 - j} \\
 &= C r^{3 - \ell} \lvert \lambda \rvert \bigg\lvert x - r \frac{x}{\lvert x \rvert} \bigg\rvert^{- 1} 
 &&+ C \max_{0 \leq j \leq \ell - 4} r^{j - \ell + 5} \lvert \lambda \rvert \bigg\lvert x - r \frac{x}{\lvert x \rvert} \bigg\rvert^{- 3 - j}.
\end{alignat*}
For the second inequality, we used~\eqref{Eq: Difference Helmholtz} with $\ell = 3$ to estimate the first summand and Lemma~\ref{Lem: Estimates on Helmholtz fundamental solution} to estimate the second summand. Choosing $r = \lvert x \rvert / 2$ immediately gives the desired estimate.
\medbreak
\noindent \textbf{Step 3: the case $\ell = 3$ and $\lvert \lambda \rvert \lvert x \rvert^2 \leq 1 / 2$.} This case uses direct calculations relying on the asymptotic expansion of $H_0^{(1)} (z)$ with $z = k \lvert x \rvert$. We omit the technical calculations from this section and instead carry them out in Appendix~\ref{sec:A1}.
\end{proof}

\begin{remark}
In the situation of Lemma~\ref{Lem: Difference Helmholtz}, one can show by considering asymptotic expansions that, for $|\lambda||x| \leq 1/2$ and $\ell \in \{1,2\}$, the  following holds with $C > 0$ depending only on $\ell$ and $\theta$
\begin{align*}
 \big\lvert \nabla^{\ell}_x \big\{ G(x ; \lambda) - G(x ; 0) \big\} \big\rvert 
 \leq C \lvert \lambda \rvert \lvert x \rvert^{2 - \ell} \Big\{ |\log(|\lambda| |x|^2) | + 1 \Big\}.
\end{align*}
This, however, will not be relevant in this work.
\end{remark}

Now, we are in the position to study fundamental solutions to the \emph{Stokes resolvent problem}
\begin{align}\label{Eq: Stokes resolvent problem without domain}
\left\{ \begin{aligned}
 \lambda u - \Delta u + \nabla \phi &= 0, \\
 \divergence(u) &= 0
\end{aligned} \right.
\end{align}
in $\IR^2$ with $\lambda \in \S_{\theta}$. 
The fundamental solutions to the scalar Helmholtz equation~\eqref{Eq: Scalar Helmholtz} and the Laplace equation form the building blocks for solutions to the Stokes resolvent problem~\eqref{Eq: Stokes resolvent problem without domain}. 
The velocity field comes from the matrix of fundamental solutions with pole at the origin
\begin{align*}
 \Gamma_{\alpha \beta} (x ; \lambda) \coloneqq G(x ; \lambda) \delta_{\alpha \beta} - \frac{1}{\lambda} \partial_{x_{\alpha}} \partial_{x_{\beta}} \big\{ G(x ; \lambda) - G(x ; 0) \big\}, \quad \alpha , \beta \in \{1 , 2\}.
\end{align*}
For the pressure, we define the vector of fundamental solutions
\begin{align*}
 \Phi_{\beta} (x) \coloneqq - \partial_{x_{\beta}} G(x ; 0) = \frac{x_{\beta}}{2 \pi \lvert x \rvert^2}, \quad \beta \in \{ 1 , 2 \}.
\end{align*}

Using that $\Delta_x G(x ; \lambda) = \lambda G(x ; \lambda)$ in $\IR^2 \setminus \{ 0 \}$, one immediately verifies that, in $\IR^2 \setminus \{ 0 \}$ and for all $\beta \in \{ 1 , 2 \}$, the functions $\Gamma_{\alpha \beta}$ and $\Phi_{\beta}$ solve
\begin{align}
\label{Eq: Fundamental solutions solve equation}
\left\{ \begin{aligned}
 (\lambda - \Delta_x) \Gamma_{\alpha \beta} (x ; \lambda) + \partial_{x_{\alpha}} \Phi_{\beta} (x) &= 0, \quad \text{for } \alpha \in \{ 1 , 2 \}, \\
 \partial_{x_{\alpha}} \Gamma_{\alpha \beta} (x ; \lambda) &= 0.
 \end{aligned} \right.
\end{align}
Note that the summation convention was used in the last line. We now present a counterpart to Lemma~\ref{Lem: Estimates on Helmholtz fundamental solution} but for the matrix of fundamental solutions $\Gamma (x ; \lambda) \coloneqq (\Gamma_{\alpha \beta} (x ; \lambda))_{\alpha , \beta = 1}^2$.

\begin{theorem}\label{Thm: Derivative of Gamma}
Let $\lambda \in \S_{\theta}$. Then, for all $\ell \geq 1$, there exists $C > 0$ depending only on $\theta$ and $\ell$ such that
\begin{align*}
 \lvert \nabla_x^{\ell} \Gamma (x ; \lambda) \rvert \leq \frac{C}{(1 + \lvert \lambda \rvert \lvert x \rvert^2) \lvert x \rvert^{\ell}}\,.
\end{align*}
\end{theorem}

\begin{proof}
In the case $\lvert \lambda \rvert \lvert x \rvert^2 > 1 / 2$, the estimate immediately follows by Lemma~\ref{Lem: Estimates on Helmholtz fundamental solution} since
\begin{align*}
 \lvert \nabla_x^{\ell} \Gamma(x ; \lambda) \rvert \leq \lvert \nabla_x^{\ell} G(x ; \lambda) \rvert + \frac{1}{\lvert \lambda \rvert} \, \lvert \nabla_x^{\ell + 2} G(x ; \lambda) \rvert + \frac{1}{\lvert \lambda \rvert} \, \lvert \nabla_x^{\ell + 2} G(x ; 0) \rvert \leq \frac{C}{(1 + \lvert \lambda \rvert \lvert x \rvert^2) \lvert x \rvert^{\ell}}\,.
\end{align*}
In the case $\lvert \lambda \rvert \lvert x \rvert^2 \leq 1 / 2$, Lemma~\ref{Lem: Estimates on Helmholtz fundamental solution} and Lemma~\ref{Lem: Difference Helmholtz} yield the estimate via
\begin{align*}
 \lvert \nabla_x^{\ell} \Gamma(x ; \lambda) \rvert \leq \lvert \nabla_x^{\ell} G(x ; \lambda) \rvert + \frac{1}{\lvert \lambda \rvert} \, \big\lvert \nabla_x^{\ell + 2} \big\{ G(x ; \lambda) - G(x ; 0) \big\} \big\rvert 
 \leq \frac{C}{(1 + \lvert \lambda \rvert \lvert x \rvert^2) \lvert x \rvert^{\ell}}\,. &\qedhere
\end{align*}
\end{proof}

If $\lambda = 0$, the Stokes resolvent problem~\eqref{Eq: Stokes resolvent problem without domain} becomes just the \emph{Stokes problem} in $\IR^2$
\begin{align*}
\left\{ \begin{aligned}
 - \Delta u + \nabla \phi &= 0, \\
 \divergence(u) &= 0.
 \end{aligned}\right. 
\end{align*}
Whereas the fundamental solution for the pressure is maintained, the matrix of fundamental solutions for the velocity field with pole at the origin is given by
\begin{align*}
 \Gamma_{\alpha \beta} (x ; 0) \coloneqq \frac{1}{4 \pi} \bigg\{ - \delta_{\alpha \beta} \log (\lvert x \rvert) + \frac{x_{\alpha} x_{\beta}}{\lvert x \rvert^2} \bigg\}, \quad \alpha, \beta \in \{1, 2\}.
\end{align*}
The following theorem presents the counterpart to Lemma~\ref{Lem: Difference Helmholtz} but for the fundamental solutions for the velocity field. Its technical proof, relying on direct computations of the fundamental solutions involved, is omitted from this section and presented in Appendix~\ref{sec:A2}.

\begin{theorem}
\label{Thm: Comparison with stationary case}
Let $\lambda \in \S_{\theta}$. Suppose that $\lvert \lambda \rvert \lvert x \rvert^2 \leq 1 / 2$. Then there exists $C > 0$ depending only on $\theta$ such that
\begin{align*}
 \big\lvert \nabla_x \big\{ \Gamma(x ; \lambda) - \Gamma(x ; 0) \big\} \big\rvert \leq C \lvert \lambda \rvert \lvert x \rvert \, \big\lvert \log(\lvert \lambda \rvert \lvert x \rvert^2) \big\rvert.
\end{align*}
\end{theorem}

\subsection{A sketch of the method of single and double layer potentials}
\label{Sec: Single and double layer potentials}

Theorem~\ref{Thm: Dirichlet} is proved by employing the method of single and double layer potentials. As the estimates on the fundamental solutions established in Subsection~\ref{Sec: Properties of the matrix of fundamental solutions to the Stokes resolvent problem} allow to literally follow the lines of the three and higher dimensional case, see~\cite[Sect.~3--5]{Shen}, we only sketch the philosophical pillars of the proof.

Let $\Xi \subset \IR^2$ denote a bounded Lipschitz domain with connected boundary. Fix $\theta \in (0 , \pi / 2)$ and let $\lambda \in \S_{\theta}$. For $f \in \L^p (\partial \Xi ; \IC^2)$, $1 < p < \infty$, define the \textit{single layer potential for the velocity field} $u \coloneqq \cS_{\lambda} (f)$ by
\begin{equation}\label{Eq: Single layer potential}
 (\cS_{\lambda} (f))_{j} (x) \coloneqq \int_{\partial \Xi} \Gamma_{j k} (x - y ; \lambda) f_k (y) \, \d \sigma (y), \quad j \in \{1,2\}.
\end{equation}
Notice that, throughout this section, we tacitly sum over repeated indices. The corresponding \textit{single layer potential for the pressure} $\phi \coloneqq \cS_{\Phi} (f)$ is given by
\begin{align*}
 \cS_{\Phi} (f) (x) \coloneqq \int_{\partial \Xi} \Phi_k (x - y) f_k (y) \, \d \sigma(y).
\end{align*}
By the smoothness of $\Gamma(\, \cdot \, ; \lambda)$ and the vector $\Phi = (\Phi_1 , \Phi_2)$ together with~\eqref{Eq: Fundamental solutions solve equation}, one directly verifies that $u$ and $\phi$ are smooth in $\IR^2 \setminus \partial \Xi$ and that $(u , \phi)$ solves the Stokes resolvent problem~\eqref{Eq: Stokes resolvent problem without domain} in $\IR^2 \setminus \partial \Xi$.

Another way of defining solutions to the Stokes resolvent problem~\eqref{Eq: Stokes resolvent problem without domain} in $\IR^2 \setminus \partial \Xi$ is by defining the so-called double layer potentials. 
To introduce these double layer potentials, let $n$ denote the outward unit normal to $\partial \Xi$ and let, for a vector field $v$ and a scalar function $\psi$, the \emph{conormal derivative} of $v$ and $\psi$ to $\partial \Xi$ be defined by
\begin{align*}
 \partial_{\nu} ( v , \psi ) \coloneqq \partial_n v - \psi n.
\end{align*}
The \textit{double layer potential for the velocity field} $u \coloneqq \cD_{\lambda} (f)$ is now given by the convolution of $f$ with a conormal derivative 
\begin{align}\label{Eq: Double layer potential}
 (\cD_{\lambda} (f))_j (x) &\coloneqq \int_{\partial \Xi} \big\{ \partial_{\nu} \big( \Gamma_{j \cdot} (\cdot - x ; \lambda) , \Phi_j (\cdot - x) \big) \big\}_k(y) f_k (y) \, \d \sigma (y) \\
 &\phantom{:}= \int_{\partial \Xi} \Big\{ \big\{ \partial_{y_i} \Gamma_{j k} (y - x ; \lambda) \big\} n_i (y) - \Phi_j (y - x) n_k (y) \Big\} f_k (y) \, \d \sigma (y).\nonumber
\end{align}
Here, we regarded $\Gamma_{j \cdot} (\, \cdot \, ; \lambda)$ as the vector $(\Gamma_{j k} (\, \cdot \, ; \lambda))_{k = 1}^2$. The corresponding \textit{double layer potential for the pressure} $\phi \coloneqq \cD_{\Phi} (f)$ is defined as
\begin{align*}
 \cD_{\Phi} (f) (x) \coloneqq \partial_{x_i} \partial_{x_k} \int_{\partial \Xi} G(y - x ; 0) n_i (y) f_k (y) \, \d \sigma (y) + \lambda \int_{\partial \Xi} G(y - x ; 0) n_k (y) f_k (y) \, \d \sigma (y).
\end{align*}
Again, one shows that $u$ and $\phi$ are smooth in $\IR^2 \setminus \partial \Xi$ and that $(u , \phi)$ solves~\eqref{Eq: Stokes resolvent problem without domain} in $\IR^2 \setminus \partial \Xi$.

Recall that, if one would replace in the definition of the single layer potential the fundamental solution by the Green's function, then the single layer potential operators would correspond to solving inhomogeneous Neumann boundary value problems with boundary value $f$. 
Moreover, in this situation, the double layer potential operators would correspond to solving inhomogeneous Dirichlet boundary value problems with boundary value $f$. 
As, unlike for the Green's functions, explicit estimates on the fundamental solutions are known from Subsection~\ref{Sec: Properties of the matrix of fundamental solutions to the Stokes resolvent problem}, one preferably works with the fundamental solutions $\Gamma(x; \lambda)$ and $\Phi(x)$. 
This, however, produces an error in the sense that, not even formally, neither the conormal derivative of $\cS_{\lambda} (f)$ and $\cS_{\Phi} (f)$ to $\partial \Xi$ coincides with $f$, nor the trace of $\cD_{\lambda} (f)$ to $\partial \Xi$ coincides with $f$. The main goal of the method of single and double layer potentials is then the following:
\medbreak
\begin{center}
\begin{tabular}{p{0.9\textwidth}}
\textit{Find a one-to-one correspondence between $f$ and $\partial_{\nu} (\cS_{\lambda} (f) , \cS_{\Phi} (f))|_{\partial \Xi}$ and also between~$f$ and $\cD_{\lambda} (f)|_{\partial \Xi}$ for $f$ in appropriate subspaces of $\L^2 (\partial \Xi ; \IC^2)$. Furthermore, derive ``good'' quantitative estimates between the involved objects.
}
\end{tabular}
\end{center}
\medbreak
The formulation of this goal is very imprecise, for it is not clear what the \emph{trace to $\partial \Xi$} should even mean. 
The apparent choice of non-tangential convergence, see Definition~\ref{Def: Nontangential maximal functions}, to define a trace to $\partial \Xi$ bears the problem of introducing a jump across the boundary $\partial \Xi$ due to a discrepancy between the non-tangential limit inside of $\Xi$ and outside of $\overline{\Xi}$, see~\cite[Sect.~3]{Shen}.
Using a non-tangential limit, see Definition~\ref{Def: Nontangential maximal functions}, as a notion for the trace to $\partial \Xi$ bears the problem of a discrepancy between the non-tangential limit inside of $\Xi$ and outside of $\overline{\Xi}$, see~\cite[Sect.~3]{Shen}.

If this main goal can be answered satisfactory, then, for a given boundary value $g$, the solutions $u$ and $\phi$ to the Dirichlet problem~\eqref{Dir} will be given by $u = \cD_{\lambda} (f)$ and $\phi = \cD_{\Phi} (f)$, where $f$ is chosen in such a way that the non-tangential limit of $u$ inside of $\Xi$ coincides with $g$.

To make this more precise, relying on the estimates established in Subsection~\ref{Sec: Properties of the matrix of fundamental solutions to the Stokes resolvent problem}, one can prove along the lines of~\cite[Sect.~3]{Shen} the following lemma. But, first, recall the definitions of the non-tangential maximal functions in Definition~\ref{Def: Nontangential maximal functions}.

\begin{lemma}
\label{Lem: Properties single layer potential}
Let $1 < p < \infty$ and $f \in \L^p (\partial \Xi ; \IC^2)$. There exists a constant $C > 0$ depending only on $\theta$, $p$, and the Lipschitz character of $\Xi$ such that $u$ and $\phi$ given by $u = \cS_{\lambda} (f)$ and $\phi = \cS_{\Phi} (f)$ satisfy
\begin{align*}
 \| (\nabla u)^*_\interior \|_{\L^p (\partial \Xi)} + \| (\nabla u)^*_\exterior \|_{\L^p (\partial \Xi)} + \| (\phi)^*_\interior \|_{\L^p (\partial \Xi)} + \| (\phi)^*_\exterior \|_{\L^p (\partial \Xi)} \leq C \| f \|_{\L^p (\partial \Xi ; \IC^2)}.
\end{align*}
Furthermore, the non-tangential limits of $\partial_j u_i$ for $i , j \in \{1 , 2\}$ and $\phi$ from inside of $\Xi$ and outside of $\overline{\Xi}$ exist and are given by
\begin{align*}
 (\partial_j u_i)_{\pm} (x) &= \pm \frac{1}{2} \big\{ n_j (x) f_i (x) - n_i (x) n_j (x) n_k (x) f_k (x) \big\} + \pv \int_{\partial \Xi} \big\{ \partial_{x_j} \Gamma_{i k} (x - y ; \lambda) \big\} f_k (y) \, \d \sigma (y) \\
 \phi_{\pm} (x) &= \mp \frac{1}{2} n_k (x) f_k (x) + \pv \int_{\partial \Xi} \Phi_k (x - y) f_k (y) \, \d \sigma (y)
\end{align*}
for almost every $x \in \partial \Xi$. The subscripts $+$ and $-$ indicate non-tangential limits taken inside $\Xi$ and outside $\overline{\Xi}$, respectively.
\end{lemma}

\begin{remark}
We remark at this point that Shen also derived bounds of the form
\begin{align*}
 \lvert \lambda \rvert^{1 / 2} \| (u)_\interior^* \|_{\L^p (\partial \Xi)} + \lvert \lambda \rvert^{1 / 2} \| (u)_\exterior^* \|_{\L^p (\partial \Xi)} \leq C \| f \|_{\L^p (\partial \Xi ; \IC^d)}
\end{align*} 
which are based on Shen's version of Lemma~\ref{Lem: Estimates on Helmholtz fundamental solution} which, in three and more dimensions, also holds in the case $\ell = 0$ with no restriction on $x$, cf.~\cite[Lem.~2.1]{Shen}. As this estimate is not needed in the derivation of the required estimates, we do not pursue this issue any further. Let us mention that $(u)_\interior^*$ and $(u)_\exterior^*$ lie in $\L^p (\partial \Xi)$ anyway whenever $(\nabla u)_\interior^*$ and $(\nabla u)_\exterior^*$ lie in $\L^2 (\partial \Xi)$ by~\cite[Prop.~8.1.3]{Shen_lectures}.
\end{remark}

Having Lemma~\ref{Lem: Properties single layer potential} at one's disposal, one can calculate non-tangential limits of the conormal derivative of $u = \cS_{\lambda} (f)$ and $\phi = \cS_{\Phi} (f)$ from inside $\Xi$ and outside $\overline{\Xi}$. 
The result reads
\begin{align}
\label{Eq: Conormal single layer}
 \big( \partial_{\nu} (u , \phi) \big)_{\pm} = \Big( \pm \frac{1}{2} \Id + \, \cK_{\lambda} \Big) f
\end{align}
for some singular integral operator $\cK_{\lambda}$ that is bounded on $\L^p (\partial \Xi ; \IC^2)$ for all $1 < p < \infty$. 
In particular, as a consequence of Theorem~\ref{Thm: Derivative of Gamma}, this operator satisfies the estimate
\begin{align*}
 \| \cK_{\lambda} f \|_{\L^p (\partial \Xi ; \IC^2)} \leq C \| f \|_{\L^p (\partial \Xi ; \IC^2)}
\end{align*}
for all $f \in \L^p (\partial \Xi ; \IC^2)$ and some constant $C > 0$ depending only on $\theta$, $p$, and the Lipschitz character of $\Xi$. \par
Since the double layer potential for the velocity field~\eqref{Eq: Double layer potential} is built up by derivatives of the single layer potentials for the velocity field and the pressure, one can use Lemma~\ref{Lem: Properties single layer potential} to prove the following result.

\begin{lemma}
\label{Lem: Properties double layer potential}
Let $1 < p < \infty$ and $f \in \L^p (\partial \Xi ; \IC^2)$. There exists a constant $C > 0$ depending only on $\theta$, $p$, and the Lipschitz character of $\Xi$ such that $u$ given by $u = \cD_{\lambda} (f)$ satisfies
\begin{align*}
 \| (u)_\interior^* \|_{\L^p (\partial \Xi)} + \| (u)_\exterior^* \|_{\L^p (\partial \Xi)} \leq C \| f \|_{\L^p (\partial \Xi ; \IC^2)}.
\end{align*}
\end{lemma}

Besides the estimates on the interior and exterior non-tangential maximal functions, one is also in the position to evaluate the non-tangential limits from inside of $\Xi$ and outside of $\overline{\Xi}$ of $u = \cD_{\lambda} (f)$. In this case, the non-tangential limits are given by
\begin{equation}\label{Eq: Dirichlet limit}
 (u)_{\pm} = \Big( \mp \frac{1}{2} \Id + \, \cK_{\overline{\lambda}}^* \Big) f,
\end{equation}
where $\cK_{\overline{\lambda}}^*$ is the adjoint of the operator $\cK_{\overline{\lambda}}$ from~\eqref{Eq: Conormal single layer} but with resolvent parameter being $\overline{\lambda}$.

We shortly remark that the main thought of the proofs of the presented results so far is the following: it is known that the results hold in the case $\lambda = 0$, see Fabes, Kenig, and Verchota~\cite{Fabes_Kenig_Verchota} in the case of three and more dimensions and the works of Brown, Mitrea, and Wright~\cite{Mitrea_Wright} and Brown, Perry, and Shen \cite{Brown_Perry_Shen} in the case of two and more dimensions. 
For example, to derive statements for $\nabla \cS_{\lambda} (f)$, one uses the decomposition
\begin{align*}
 \nabla \cS_{\lambda} (f) =  \nabla \cS_0 (f) + \big(\nabla \cS_{\lambda} (f) - \nabla \cS_0 (f) \big) .
\end{align*}
Then, one uses the desired property for $\nabla \cS_0 (f)$ which is already known by~\cite{Fabes_Kenig_Verchota, Mitrea_Wright} and derives the same property for $\nabla \cS_{\lambda} (f) - \nabla \cS_0 (f)$ as, for the difference, the involved integral kernel has a very weak singularity due to Theorem~\ref{Thm: Comparison with stationary case}.

A consequence of the discussion above and~\eqref{Eq: Dirichlet limit} is that Theorem~\ref{Thm: Dirichlet} follows, once we find a suitable subspace $U \subset \L^2 (\partial \Xi ; \IC^2)$ such that $(- \frac{1}{2} \Id + \, \cK_{\overline{\lambda}}^*)|_U$ is invertible and has range $\L^2_n (\partial \Xi)$. Moreover, to prove that the constant in the estimate~\eqref{Eq: Stability estimate} of the non-tangential maximal function $(u)^*_\interior$ only depends on $\theta$, the Lipschitz character of $\Xi$, and uniformly on the diameter of $\Xi$ one has to show---having Lemma~\ref{Lem: Properties double layer potential} in mind---the existence of $C > 0$ depending only on $\theta$, the Lipschitz character of $\Xi$, and $\alpha, \beta > 0$ (where $\alpha$ and $\beta$ satisfy $\alpha \leq \diam(\Xi) \leq \beta$) such that, for all $f \in U$, it holds that
\begin{align}
\label{Eq: Lower estimate on potential}
 \| f \|_{\L^2 (\partial \Xi ; \IC^2)} \leq C \| (- \tfrac{1}{2} \Id + \, \cK_{\overline{\lambda}}^*) f \|_{\L^2 (\partial \Xi ; \IC^2)}.
\end{align}

This analysis is performed in~\cite[Sect.~5]{Shen} and fundamentally relies on Fredholm theory. 
Indeed, it is known~\cite[Eq.~(5.165)]{Mitrea_Wright} that $- \frac{1}{2} \Id + \, \cK_0$ is a Fredholm operator with index zero. Now, $- \frac{1}{2} \Id + \, \cK_{\overline{\lambda}}$ is a compact perturbation of $- \frac{1}{2} \Id + \, \cK_0$ since
\begin{align*}
 - \frac{1}{2} \Id + \, \cK_{\overline{\lambda}} = - \frac{1}{2} \Id + \, \cK_0 + \big( \cK_{\overline{\lambda}} - \cK_0 \big)
\end{align*}
and since $\cK_{\overline{\lambda}} - \cK_0$ is a compact operator thanks to the good kernel estimate in Theorem~\ref{Thm: Comparison with stationary case}. It follows that $- \frac{1}{2} \Id + \, \cK_{\overline{\lambda}}$ is also a Fredholm operator with index zero. 
Due to
\begin{align*}
 \Rg(- \tfrac{1}{2} \Id + \, \cK_{\overline{\lambda}}^*) = \overline{\Rg(- \tfrac{1}{2} \Id + \, \cK_{\overline{\lambda}}^*)} = \ker(- \tfrac{1}{2} \Id + \, \cK_{\overline{\lambda}})^{\perp}
\end{align*}
and
\begin{align*}
 \ker (- \tfrac{1}{2} \Id + \, \cK_{\overline{\lambda}}^*)^{\perp} = \overline{\Rg(- \tfrac{1}{2} \Id + \, \cK_{\overline{\lambda}})} = \Rg(- \tfrac{1}{2} \Id + \, \cK_{\overline{\lambda}})\,,
\end{align*}
one defines the desired space $U$ to be $U \coloneqq \Rg(- \tfrac{1}{2} \Id + \, \cK_{\overline{\lambda}})$. Furthermore, it turns out that one can compute the kernel of $- \frac{1}{2} \Id + \, \cK_{\overline{\lambda}}$ to be the linear span in $\L^2 (\partial \Xi ; \IC^2)$ of the normal vector $n$. This results in
\begin{align*}
 \Rg(- \tfrac{1}{2} \Id + \, \cK_{\overline{\lambda}}^*) = \L^2_n (\partial \Xi).
\end{align*}
The key result that allows to deduce that the kernel of $- \frac{1}{2} \Id + \, \cK_{\overline{\lambda}}$ equals the linear span of the normal vector $n$ in $\L^2 (\partial \Xi ; \IC^2)$ are so-called \textit{Rellich estimates}, see~\cite[Sect.~4]{Shen}. These estimates further allow to prove the existence of a constant $C > 0$ depending only on $\theta$, the Lipschitz character of $\Xi$, $\alpha$, and $\beta$ such that, for all $f \in \L^2_n (\partial \Xi)$, it holds that
\begin{align*}
  \| f \|_{\L^2 (\partial \Xi ; \IC^2)} \leq C \| (- \tfrac{1}{2} \Id + \, \cK_{\overline{\lambda}}) f \|_{\L^2 (\partial \Xi ; \IC^2)}.
\end{align*}
By duality, this gives the bound in~\eqref{Eq: Lower estimate on potential} with the same constant $C$.

Finally, we remark that the Rellich estimates established in~\cite[Sect.~4]{Shen} do not depend on properties of the fundamental solutions but are established by very clever integrations by parts. This approach carries over from three and more dimensions to two dimensions word by word, so we will not go into details here. 
We only mention that the resolutions of the inhomogeneous Dirichlet, Neumann, and regularity problems for the Laplacian on bounded Lipschitz domains with connected boundary were used:
for the Dirichlet problem in two space dimensions, we refer to Dahlberg's works~\cites{Dahlberg.1977, Dahlberg.1979}, in particular to \cite[Thm.~2]{Dahlberg.1979} (noticing the remark  in~\cite[p.~14]{Dahlberg.1979}), the formulations of this result in the works of Verchota~\cite[Thm.~4.7D]{Verchota_Dissertation} and \cite[Thm.~0.9D]{Verchota_Layer_Potentials}, and the recent work by Shen~\cite[Thm.~8.5.14]{Shen_lectures}.
For the Neumann problem in two space dimensions, we refer to Verchota~\cite[Cor.~3.4 and Thm.~4.2]{Verchota_Layer_Potentials}, Jerison and Kenig~\cite[Rem. (d)]{Jerison_Kenig}, and Kenig~\cite{Kenig.1980}. See also the formulation and proof in Shen~\cite[Thm.~8.5.13]{Shen_lectures}.
For the regularity problem in two space dimensions, we refer to Kenig's work~\cite[Thm.~2.1.10]{Kenig.1994} noting that the necessary ingredients \cite[Lem.~2.1.13]{Kenig.1994} and \cite[Lem.~1.10.8]{Kenig.1994} also hold in the two-dimensional case. We also refer to Shen's monograph~\cite[Thm.~8.5.15]{Shen_lectures}.

With this being said, we conclude our review of Shen's proof on the resolution of the $\L^2$-Dirichlet problem for the Stokes resolvent problem.

\section{Functional analytic properties of the Stokes operator}

\noindent In this section, we will prove several functional analytic properties of the Stokes operator.

\subsection{Resolvent estimates and maximal regularity}
\label{Sec: Resolvent estimates and maximal regularity}

We start this section with a definition. For this purpose, recall the definition of the sector $\S_{\theta} \coloneqq \{ z \in \IC \setminus \{ 0 \} : \lvert \arg(z) \rvert < \pi - \theta \}$, $\theta \in (0 , \pi / 2)$.

\begin{definition}
\label{Def: Sectorial and R-sectorial}
Let $X$ and $Y$ denote Banach spaces over $\mathbb{C}$ and $B \colon \dom(B) \subset X \to X$ a closed linear operator. 
\begin{enumerate}[label=\normalfont{(\roman*)}]
 \item The operator $B$ is said to be \textit{sectorial} of angle $\omega \in [0 , \pi)$ if
 \begin{align*}
  \sigma (B) \subset \overline{\S_{\pi - \omega}}
 \end{align*}
 and if, for all $\omega < \theta < \pi$, the family $\{ \lambda (\lambda + B)^{-1} \}_{\lambda \in \S_{\theta}} \subset \cL(X)$ is bounded.
 \item\label{Item: R-boundedness} A family of operators $\cT \subset \cL (X, Y)$ is said to be \emph{$\cR$-bounded from $X$ to $Y$} if there exists a positive constant $C > 0$ such that, for any $k_0 \in \IN$, $(T_k)_{k = 1}^{k_0} \subset \cT$, and $(x_k)_{k = 1}^{k_0} \subset X$, the inequality
\begin{align*}
\Big\| \sum_{k = 1}^{k_0} r_k (\cdot) T_k x_k \Big\|_{\L^2 (0 , 1 ; Y)}
\leq C \Big\| \sum_{k = 1}^{k_0} r_k (\cdot) x_k \Big\|_{\L^2 (0 , 1 ; X)}
\end{align*}
holds. Here, $r_k (t) \coloneqq \mathrm{sgn} (\sin(2^k \pi t))$ are the \textit{Rademacher functions}. 
We denote with $\cR(\cT)$ the minimal constant $C > 0$ such that $\cT$ is $\cR$-bounded.

 \item The operator $B$ is said to be \textit{$\cR$-sectorial} of angle $\omega \in [0 , \pi)$ if
 \begin{align*}
  \sigma (B) \subset \overline{\S_{\pi - \omega}}
 \end{align*}
 and if, for all $\omega < \theta < \pi$, the family $\{ \lambda (\lambda + B)^{-1} \}_{\lambda \in \S_{\theta}} \subset \cL(X)$ is $\cR$-bounded.
\end{enumerate}
\end{definition}

\begin{remark}\label{Rem: R-boundedness}
\begin{enumerate}[label=\normalfont{(\roman*)}]
 \item By taking $k_0 = 1$, one sees that $\cR$-boundedness implies boundedness of a family of operators. If $X$ and $Y$ are isomorphic to a Hilbert space, then $\cR$-boundedness is equivalent to the boundedness of the family of operators, see~\cite[Rem.~3.2]{Denk_Hieber_Pruess}. \label{Item: R-boundedness implies boundedness}
 \item If $X$ is a subspace of $\L^p (\Omega ; \IC^m)$ for some $1 < p < \infty$, $m \in \IN$ and $\Omega \subset \IR^d$ Lebesgue-measurable, then there exists $C > 0$ such that, for all $k_0 \in \IN$ and $(f_k)_{k = 1}^{k_0} \subset X$, it holds that
 \begin{align*}
  \frac{1}{C} \Big\| \sum_{k = 1}^{k_0} r_k (\cdot) f_k \Big\|_{\L^2 (0 , 1 ; X)} \leq \Big\| \Big[ \sum_{k = 1}^{k_0} \lvert f_k \rvert^2 \Big]^{1 / 2} \Big\|_{\L^p (\Omega)} \leq C \Big\| \sum_{k = 1}^{k_0} r_k (\cdot) f_k \Big\|_{\L^2 (0 , 1 ; X)}.
 \end{align*}
 This means that $\cR$-boundedness in $\L^p$-spaces is equivalent to so-called \emph{square function estimates}, see~\cite[Rem.~2.9]{Kunstmann_Weis_2004}. \label{Item: Square function estimate}
 \item\label{it: Strongly Continuous SG} The operator $- B$ generates a strongly continuous bounded analytic semigroup on $X$ if and only if $B$ is densely defined and sectorial of angle $\omega \in [0 , \pi / 2)$, see~\cite[Thm.~II.4.6]{Engel_Nagel}. Moreover, this semigroup is exponentially stable if $0 \in \rho(B)$. A short calculation reveals that the condition $0 \in \rho(B)$ follows if one can show the boundedness of the family of operators $\{ (1 + \lvert \lambda \rvert) (\lambda + B)^{-1} \}_{\lambda \in \S_{\theta}}$ for $\omega < \theta < \pi$. \label{Item: Sectoriality and semigroups}
\end{enumerate}
\end{remark}

If $X$ is a $\mathrm{UMD}$-space, then the question of $\cR$-sectoriality is intimately related to the question of the maximal $\L^q$-regularity of generators $- B \colon \dom(B) \subset X \to X$ of a bounded analytic semigroup. However, we will not go into details with the definition of $\mathrm{UMD}$-spaces as the reflexive $\L^p$-spaces as well as all closed subspaces and quotient spaces of these $\L^p$-spaces have the $\mathrm{UMD}$-property, see, e.g.,~\cite{Burkholder}, and these will be the only spaces that appear in our investigation.

Recall that, for all $x \in X$ and $f \in \L^q (0 , \infty ; X)$, $1 < q < \infty$, the unique \emph{mild solution} to the abstract Cauchy problem
\begin{align}\label{Eq: abstract CP}
 \left\{ \begin{aligned}
  u^{\prime} (t) + B u (t) &= f (t), && t \geq 0, \\
  u(0) &= x
 \end{aligned} \right.\tag{ACP}
\end{align}
is given by \emph{Duhamel's formula} 
\begin{align*}
 u (t) = \e^{-t B} x + \int_0^t \e^{- (t - s) B} f(s) \, \d s \,,
\end{align*}
cf.~\cite[Prop.~3.1.16]{Arendt_Batty_Hieber_Neubrander}. 

\begin{definition}\label{Def: Maxreg}
For $1 < q < \infty$, we say that \emph{$B \colon \cD(B) \subset X \to X$ has maximal $\L^q$-regularity} if, for $x = 0$ and all $f \in \L^q (0 , \infty ; X)$, the unique mild solution to~\eqref{Eq: abstract CP} satisfies $u(t) \in \dom(B)$ for almost every $t > 0$ and $B u \in \L^q (0 , \infty ; X)$.
\end{definition}

\begin{remark}
In the case of Definition~\ref{Def: Maxreg}, $u$ is also weakly differentiable with respect to $t$ and satisfies $u^{\prime} \in \L^q (0 , \infty ; X)$. 
By employing the closed graph theorem, the mere fact that $u^{\prime}$ and $B u$ lie in $\L^q (0 , \infty ; X)$ implies the existence of $C > 0$ such that, for all $f \in \L^q (0 , \infty ; X)$, the stability estimate
\begin{align*}
 \| u^{\prime} \|_{\L^q (0 , \infty ; X)} + \| B u \|_{\L^q (0 , \infty ; X)} \leq C \| f \|_{\L^q (0 , \infty ; X)}
\end{align*}
holds. Furthermore, it is well-known, see, e.g.,~\cite[Thm.~2.1]{Dore}, that $u \in \L^q (0 , \infty ; X)$ if and only if $0 \in \rho(B)$. In this case, one has a stability estimate of the form
\begin{equation}\label{Eq: Homogeneous MaxReg}
 \| u \|_{\L^q (0 , \infty ; X)} + \| u^{\prime} \|_{\L^q (0 , \infty ; X)} + \| B u \|_{\L^q (0 , \infty ; X)} \leq C \| f \|_{\L^q (0 , \infty ; X)}.
\end{equation}
This estimate~\eqref{Eq: Homogeneous MaxReg} and with it the regularity of $u$ can be extended to mild solutions with inhomogeneous initial data $u(0) = x$ for all $x$ in the real interpolation space $(X, \cD(B))_{1 - 1/q, q}$, cf.~\cite[Prop.~2.2.3]{Tolksdorf_Dissertation}. In this case, the stability estimate generalizes to 
\begin{align*}
 \| u \|_{\L^q (0 , \infty ; X)} + \| u^{\prime} \|_{\L^q (0 , \infty ; X)} + \| B u \|_{\L^q (0 , \infty ; X)} \leq C \big( \| f \|_{\L^q (0 , \infty ; X)} + \| x \|_{(X , \cD(B))_{1 - 1/q , q}}  \big).
\end{align*}
Notice further that the property that $B$ has maximal $\L^q$-regularity is \textit{independent} of $q$, see~\cite[Thm.~4.2]{Dore} or~\cite[Thm.~7.1]{Dore_Max_Reg} and the references given there.
\end{remark}

A seminal result of Weis~\cite[Thm.~4.2]{Weis} now builds the bridge between the notions of maximal $\L^q$-regularity and $\cR$-sectoriality.
\begin{theorem}[Weis]
\label{Thm: Weis}
If $X$ is a $\mathrm{UMD}$-space and $- B \colon \dom(B) \subset X \to X$ is the generator of a bounded analytic semigroup, then $B$ has maximal $\L^q$-regularity for $1 < q < \infty$ if and only if $B$ is $\cR$-sectorial of angle $\omega$ for some $\omega \in [0 , \pi / 2)$.
\end{theorem}

We concentrate now on the Stokes operator and establish that it is $\cR$-sectorial of angle $\omega = 0$ on $\L^p_\sigma(\Omega)$ for suitable $p$. 
By virtue of Remark~\ref{Rem: R-boundedness}~\ref{Item: R-boundedness implies boundedness} and~\ref{Item: Sectoriality and semigroups} and Theorem~\ref{Thm: Weis}, this will readily prove Theorem~\ref{Thm: Resolvent}, Corollary~\ref{Cor: Analyticity}, and Theorem~\ref{Thm: Maximal regularity}. 
As the argument in three or more space dimensions is known and does not differ in the two-dimensional case, we only give a short sketch:

\subsection*{Proof of the \texorpdfstring{$\cR$}{R}-sectoriality of the Stokes operator}

The proof is decomposed into several steps.

\subsection*{Step 1: the case \texorpdfstring{$p = 2$}{p = 2}}

Let $\theta \in (0 , \pi)$. Notice that the Lax--Milgram lemma directly implies that $\S_{\theta} \cup \{ 0 \} \subset \rho(- A_2)$. Furthermore, for any $f \in \L^2 (\Omega ; \IC^2)$, define $u \coloneqq (\lambda + A_2)^{-1} \IP_2 f$. 
Then, testing the resolvent equation with $u$ followed by and applying Poincar\'e's inequality, one directly obtains the inequality
\begin{align}
\label{Eq: L2 resolvent}
 (1 + \lvert \lambda \rvert) \| u \|_{\L^2_{\sigma} (\Omega)} \leq C \| f \|_{\L^2 (\Omega ; \IC^d)}.
\end{align}
The constant $C > 0$ depends only on $d$, $\theta$, and $\diam(\Omega)$. 
By virtue of Remark~\ref{Rem: R-boundedness}~\ref{Item: R-boundedness implies boundedness}, this implies that the family $\{ (1 + \lvert \lambda \rvert) (\lambda + A_2)^{-1} \IP_2 \}_{\lambda \in \S_{\theta}} \subset \cL(\L^2 (\Omega ; \IC^2) , \L^2_{\sigma} (\Omega))$ is $\cR$-bounded. As $\IP_2$ acts as the identity on $\L^2_{\sigma} (\Omega)$, this readily settles the $\cR$-sectoriality in the case $p = 2$.

\subsection*{Step 2: reformulation into an \texorpdfstring{$\ell^2$}{l2}-valued boundedness estimate in the case \texorpdfstring{$p \geq 2$}{p >= 2}}

Let $p \geq 2$. A combination of Definition~\ref{Def: Sectorial and R-sectorial}~\ref{Item: R-boundedness} and Remark~\ref{Rem: R-boundedness}~\ref{Item: Square function estimate} reveals that the following statement is immediate, cf.~\cite[Prop.~2.3.4]{Tolksdorf_Dissertation}: \medbreak
\emph{
The family $\{ (1 + \lvert \lambda \rvert) (\lambda + A_2)^{-1} \IP_2|_{\L^p} \}_{\lambda \in \S_{\theta}}$ is $\cR$-bounded from $\L^p (\Omega ; \IC^2)$ into $\L^p_{\sigma} (\Omega)$ if and only if there exists $C > 0$ such that, for all $k_0 \in \IN$ and all $(\lambda_k)_{k = 1}^{k_0} \subset \S_{\theta}$, the operator 
\begin{align*}
 T_{(\lambda_k)_{k = 1}^{k_0}} \colon \L^p (\Omega ; \ell^2 (\IC^2)) \to \L^p (\Omega ; \ell^2 (\IC^2)), 
 \quad f = (f_k)_{k \in \IN} \mapsto 
 \begin{pmatrix} 
 (1 + \lvert \lambda_1 \rvert) (\lambda_1 + A_2)^{-1} \IP_2|_{\L^p} f_1 \\ 
 \vdots \\ (1 + \lvert \lambda_{k_0} \rvert) (\lambda_{k_0} + A_2)^{-1} \IP_2|_{\L^p} f_{k_0} \\
 0 \\ 
 \vdots 
 \end{pmatrix}
\end{align*}
is well-defined and satisfies
\begin{align}
\label{Eq: Vector-valued boundedness}
 \| T_{(\lambda_k)_{k = 1}^{k_0}} f \|_{\L^p (\Omega ; \ell^2 (\IC^2))} \leq C \| f \|_{\L^p (\Omega ; \ell^2 (\IC^2))}.
\end{align}
In other words, the family of all operators that can be formed by the procedure above is uniformly bounded.
}

\subsection*{Step 3: verification of~\texorpdfstring{\eqref{Eq: Vector-valued boundedness}}{(\ref{Eq: Vector-valued boundedness})} for certain values of \texorpdfstring{$p > 2$}{p > 2}} 
To verify~\eqref{Eq: Vector-valued boundedness}, one employs the following vector-valued version of the $\L^p$-extrapolation theorem of Shen~\cite[Thm.~4.1]{Tolksdorf_Systems}, see~\cite[Thm.~3.3]{Shen_Riesz} for the scalar-valued version.

\begin{theorem}
\label{Thm: Extrapolation of square function estimates}
Let $X$ be a Banach space, $\Upsilon \subset \IR^2$ be Lebesgue-measurable and bounded, $\cM > 0$, and let $T \in \Lop(\L^2(\Upsilon ; X))$ with $\| T \|_{\Lop(\L^2(\Upsilon ; X))} \leq \cM$.

Suppose that there exist constants $p > 2$, $R_0 > 0$, $\alpha_2 > \alpha_1 > 1$, and $\cC > 0$ such that the following statement is valid: for all balls $B(x_0 , r)$ with $0 < r < R_0$, which are either centered on $\partial \Upsilon$, i.e., $x_0 \in \partial \Upsilon$, or satisfy $B(x_0 , \alpha_2 r) \subset \Upsilon$, and all compactly supported $f \in \L^{\infty}(\Upsilon ; X)$ with $f = 0$ in $\Upsilon \cap B (x_0 , \alpha_2 r)$, the estimate
\begin{align}
\label{Eq: Weak reverse Hoelder inequality}
 \begin{aligned}
 \bigg( \frac{1}{r^2} \int_{\Upsilon \cap B(x_0 , r)} \| T f \|_X^p \, \d x \bigg)^{\frac{1}{p}} &\leq \cC \, \bigg( \frac{1}{r^2} \int_{\Upsilon \cap B(x_0 , \alpha_1 r)} \| T f \|_X^2 \, \d x \bigg)^{\frac{1}{2}}
 \end{aligned}
\end{align}
holds. 
Then, for each $2 < q < p$, the operator $T$ restricts to a bounded linear operator on $\L^q (\Upsilon ; X)$, with operator norm bounded by a constant depending on $p$, $q$, $\alpha_1$, $\alpha_2$, $\cC$, $\cM$, $R_0$, and $\diam(\Upsilon)$.
\end{theorem}

\begin{remark}
We say that an operator satisfies a \emph{weak reverse Hölder estimate} if it satisfies~\eqref{Eq: Weak reverse Hoelder inequality}. The original version of Theorem~\ref{Thm: Extrapolation of square function estimates} comes with an additional term on the right-hand side of~\eqref{Eq: Weak reverse Hoelder inequality} which will not be needed in our case.
\end{remark}

We apply Theorem~\ref{Thm: Extrapolation of square function estimates} as follows: first of all, we fix some notation regarding the bounded Lipschitz domain $\Omega \subset \IR^2$. Let $M > 0$ denote the Lipschitz constant of $\Omega$. Define for $r > 0$
\begin{align*}
 D(r) \coloneqq \big\{ (x_1 , x_2) \in \IR^2 : \lvert x_1 \rvert < r \; \text{and} \; \lvert x_2 \rvert < (1 + M) r \big\},
\end{align*}
and let $R_0 > 0$ be such that, for each $x_0 \in \partial \Omega$, there exists a Lipschitz continuous function $\eta \colon \IR \to \IR$ with $\eta (0) = 0$ and $\| \eta^{\prime} \|_{\L^{\infty} (\IR)} \leq M$ such that, for all $0 < r < R_0$, one has after a possible rotation that
\begin{align*}
 D(r) \cap [\Omega - \{ x_0 \}] = \big\{ (y_1 , y_2) \in \IR^2 : \lvert y_1 \rvert < r \; \text{and} \; \eta(y_1) < y_2 < (1 + M) r \big\} =: D_{\eta} (r)
\end{align*}
and
\begin{align*}
 D(r) \cap [\partial \Omega - \{ x_0 \}] = \big\{ (y_1 , y_2) \in \IR^2 : \lvert y_1 \rvert < r \; \text{and} \; y_2 = \eta(y_1) \big\} =: I_{\eta} (r).
\end{align*}
Notice that
\begin{align}
\label{Eq: Balls and cylinders}
 B(0 , r) \subset D(r) \quad \text{and} \quad D (2 r) \subset B(0 , \alpha_1 r)
\end{align}
with
\begin{align*}
 \alpha_1 \coloneqq 2 \sqrt{1 + (1 + M)^2}.
\end{align*}
Furthermore, we remark that the \emph{Lipschitz cylinders} $D_{\eta} (r)$ are Lipschitz domains with Lipschitz constant comparable to $M$. Additionally, the Lipschitz character is uniform with respect to all parameters if the radius $r$ is bounded from below. This was discussed in detail, e.g., in~\cite[Lem.~1.3.25]{Tolksdorf_Dissertation}.

Now, choose $X \coloneqq \ell^2 (\IC^2)$, $\Upsilon \coloneqq \Omega$, and $p_0 \coloneqq 4$. Let further $\theta \in (0 , \pi / 2)$, $k_0 \in \IN$, and let $(\lambda_k)_{k = 1}^{k_0} \subset \S_{\theta}$. 
For $T$ being defined by $T \coloneqq T_{(\lambda_k)_{k = 1}^{k_0}}$, we know by Steps~1 and~2 that $T$ is bounded on $\L^2 (\Omega ; \ell^2 (\IC^2))$ and that the operator norm is bounded by the constant $C$ from~\eqref{Eq: L2 resolvent}. As this constant is uniform with respect to $k_0$ and all choices of $(\lambda_k)_{k = 1}^{k_0} \subset \S_{\theta}$, we choose $\cM \coloneqq C$.

Now, if we can establish the validity of~\eqref{Eq: Weak reverse Hoelder inequality} uniformly in those parameters as well, the family of all such operators $T$ restricts to a bounded family of operators in $\L^q (\Omega ; \ell^2 (\IC^2))$ for all $q$ subject to $2 \leq q < p$. Here, the number $p$ still has to be fixed.

We concentrate on verifying~\eqref{Eq: Weak reverse Hoelder inequality} only for points in $x_0 \in \partial \Omega$ as this is the more demanding task. Define $\alpha_2 \coloneqq 3 \sqrt{1 + (1 + M)^2}$ and let $0 < 2 r < R_0$. Finally, let $f \in \L^{\infty} (\Omega ; \ell^2 (\IC^2))$ have compact support and satisfy $f = 0$ in $\Omega \cap B(x_0 , \alpha_2 r)$.

For $1 \leq k \leq k_0$ define $u_k \coloneqq (\lambda_k + A_2)^{-1} \IP_2 f_k$ and notice that $u_k \in \W^{1,2}_{0 , \sigma} (\Omega)$. 
We remark that, by virtue of~\cite[Cor.~5.7]{Mitrea_Monniaux}, $u_k$ is even H\"older continuous in $\overline{\Omega}$. 
Thus, e.g., by~\cite[Thm.~III.2.1.1]{Sohr} there exists a unique pressure $\pi_k \in \L^2 (\Omega)$ with average zero such that
\begin{align*}
\left\{ \begin{aligned}
 \lambda_k u _k - \Delta u_k + \nabla \pi_k &= \IP_2 f_k && \text{in } \Omega,\\
 \divergence(u_k) &= 0 && \text{in } \Omega, \\
 u_k &= 0 && \text{on } \partial \Omega.
\end{aligned} \right.
\end{align*}
Using the Helmholtz decomposition of $\L^2(\Omega; \IC^2)$ in order to write $\IP_2 f_k = f_k - \nabla h_k$ with some $h_k \in \mathrm{G}_2(\Omega)$, we find, since $f_k$ vanishes in $\Omega \cap B(x_0 , \alpha_2 r)$, with $\phi_k \coloneqq \pi_k + h_k$ that
\begin{align}
\label{Eq: Localized Stokes resolvent}
\left\{ \begin{aligned}
 \lambda_k u _k - \Delta u_k + \nabla \phi_k &= 0 && \text{in } \Omega \cap B(x_0 , \alpha_2 r),\\
 \divergence(u_k) &= 0 && \text{in } \Omega \cap B(x_0 , \alpha_2 r), \\
 u_k &= 0 && \text{on } \partial \Omega \cap B(x_0 , \alpha_2 r).
\end{aligned} \right.
\end{align}
Inner regularity, see, e.g.,~\cite[Thm.~IV.4.3]{Galdi}, now implies that the solutions $u_k$ and $\phi_k$ are smooth in $\Omega \cap B(x_0 , \alpha_2 r)$.

By translation, rotation, and rescaling, we may assume that $x_0 = 0$, that $r = 1$, and that $u_k$ and $\phi_k$ solve~\eqref{Eq: Localized Stokes resolvent} in $D_{\eta} (2)$. Let $1 < s < 2$, and let $g_{k , s} \coloneqq u_k|_{\partial D_{\eta} (s)}$. 
By Theorem~\ref{Thm: Dirichlet}, there exist $v_{k , s}$ and $\vartheta_{k , s}$ solving the Stokes resolvent problem with resolvent parameter $\lambda_k$ in the Lipschitz domain $\Xi \coloneqq D_{\eta} (s)$ and boundary data $g_{k , s}$. Notice that the boundary datum in fact is the trace of $u_k$ itself. Moreover, since $u_k$ is H\"older continuous in $\overline{D_{\eta} (s)}$, its non-tangential maximal function from the inside lies in $\L^2 (\partial D_{\eta} (s))$ so that, by uniqueness of the solution, we must have that $u_k = v_{k , s}$. 
Finally, notice for every complex-valued measurable function $h$ the validity of the estimate
\begin{align}\label{eq: Estimate by non-tangential maximal function}
 \bigg( \int_{D_{\eta} (s)} \lvert h \rvert^4 \, \d x \bigg)^{\frac{1}{4}} \leq C \, \bigg( \int_{\partial D_{\eta} (s)} \lvert (h)_\interior^* \rvert^2 \, \d \sigma (x) \bigg)^{\frac{1}{2}},
\end{align}
where the constant $C > 0$ only depends on the Lipschitz constant $M$, see~\cite[Lem.~3.3]{Wei_Zhang}. 
We remark that the result in~\cite{Wei_Zhang} is only formulated for $d \geq 3$, but the proof carries over word for word to the case of $d = 2$.
Thus, by virtue of inequality~\eqref{eq: Estimate by non-tangential maximal function} and Theorem~\ref{Thm: Dirichlet}, we find that
\begin{align*}
 \bigg( \int_{D_{\eta} (1)} \Big[ \sum_{k = 1}^{k_0} \lvert (1 + \lvert \lambda_k \rvert ) u_k \rvert^2 \Big]^{\frac{4}{2}} \, \d x \bigg)^{\frac{2}{4}} &\leq C \int_{\partial D_{\eta} (s)} \Big\{ \Big( \Big[ \sum_{k = 1}^{k_0} \lvert (1 + \lvert \lambda_k \rvert) u_k \rvert^2 \Big]^{\frac{1}{2}} \Big)^*_\interior \Big\}^2 \, \d \sigma(x) \\
 &\leq C \int_{\partial D_{\eta} (s)} \sum_{k = 1}^{k_0} \lvert (1 + \lvert \lambda_k \rvert ) (u_k)_\interior^* \rvert^2 \, \d \sigma(x) \\
 &\leq C \int_{\partial D_{\eta} (s) \setminus I_{\eta} (s)} \sum_{k = 1}^{k_0}  \lvert (1 + \lvert \lambda_k \rvert ) u_k \rvert^2 \, \d \sigma(x).
\end{align*}
Notice that the constant $C$ is independent of $s$. Integrating this inequality over $s \in (1 , 2)$, the co-area formula yields
\begin{align*}
 \bigg( \int_{D_{\eta} (1)} \Big[ \sum_{k = 1}^{k_0} \lvert ( 1 + \lvert \lambda_k \rvert ) u_k \rvert^2 \Big]^{\frac{4}{2}} \, \d x \bigg)^{\frac{2}{4}} \leq C \int_{D_{\eta} (2)} \Big[ \sum_{k = 1}^{k_0} \lvert (1 + \lvert \lambda_k \rvert ) u_k \rvert^2 \Big]^{\frac{2}{2}} \, \d x,
\end{align*}
which is, after taking the square root and using the inclusion relations~\eqref{Eq: Balls and cylinders}, already the weak reverse H\"older estimate~\eqref{Eq: Weak reverse Hoelder inequality} with $p = p_0$.

The same strategy can be employed to establish this weak reverse H\"older estimate on balls $B (x_0 , r)$ such that $B(x_0 , \alpha_2 r) \subset \Omega$. Having these weak reverse H\"older estimates at hand, one uses the \emph{Vitali covering lemma} to deduce the validity of these weak reverse H\"older estimates on all balls $B (x_0 , r)$ such that $\Omega \cap B(x_0 , \alpha_2 r) \neq \emptyset$, cf.~the proof of~\cite[Lem.~4.2]{Tolksdorf_Systems}. Afterwards, one uses their self-improving property, see, e.g.,~\cite[Thm.~6.38]{Giaquinta_Martinazzi}, to deduce the validity of~\eqref{Eq: Weak reverse Hoelder inequality} with $p = p_0 + \eps$ for some $\eps > 0$ depending only on the Lipschitz character of $\Omega$ and $\theta$. As all parameters are uniform with respect to $(\lambda_k)_{k = 1}^{k_0}$, we conclude that~\eqref{Eq: Vector-valued boundedness} holds for all $2 < p < 4 + \eps$ with a uniform constant $C > 0$.

\subsection*{Step 4: the case \texorpdfstring{$p < 2$}{p < 2}}

Let $1 < p < 2$ be such that its H\"older conjugate exponent $p^{\prime}$ satisfies~\eqref{Eq: Helmholtz two dimensions}. Since $A_p$ is defined to be the adjoint to $A_{p^{\prime}}$, the $\cR$-sectoriality follows by duality, see~\cite[Lem.~3.1]{Kalton_Weis}.

\subsection*{Step 5: conclusion}

Having established the $\cR$-sectoriality of $A_p$ on $\L^p_{\sigma} (\Omega)$ for $2 \leq p < 4 + \eps$ via square function estimates in Steps~1,~2, and~3 as well for $(4 + \eps)^{\prime} < p \leq 2$ via duality in Step~4, we conclude the proof of Theorem~\ref{Thm: Resolvent}, Corollary~\ref{Cor: Analyticity}, and Theorem~\ref{Thm: Maximal regularity}. \qed

\subsection{Boundedness of the \texorpdfstring{$\H^{\infty}$}{H-infinity}-calculus}

To introduce the boundedness of the $\H^{\infty}$-calculus, define for $\theta \in (0 , \pi)$ the so-called \emph{Dunford--Riesz class}
\begin{align*}
 \H_0^{\infty} (\S_{\theta}) \coloneqq \bigg\{ f \colon \S_{\theta} \to \IC : f \text{ holomorphic and } \exists \eps , C > 0 \; \forall z \in \S_{\theta} \, : \, \lvert f (z) \rvert \leq \frac{C \lvert z \rvert^{\eps}}{(1 + \lvert z \rvert)^{2 \eps}} \bigg\}.
\end{align*}
Let $B \colon \dom(B) \subset X \to X$ be a sectorial operator of angle $\omega \in [0 , \pi)$ on a Banach space $X$ over the complex field. Then for any $\theta \in (0 , \pi - \omega)$ and $\vartheta \in (\theta , \pi - \omega)$ one defines for $f \in \H_0^{\infty} (\S_{\theta})$
\begin{align*}
 f(B) \coloneqq \frac{1}{2 \pi \ii} \int_{\partial \S_{\vartheta}} f(\lambda) (\lambda - B)^{-1} \, \d \lambda. 
\end{align*}
Here, the path $\partial \S_{\vartheta}$ is understood to be run through in the counterclockwise sense. Using the sectoriality of $B$ and the fact that $f \in \H_0^{\infty} (\S_{\theta})$ it is clear that $f(B) \in \Lop (X)$. If $B$ is densely defined and has a dense range, then the question of whether there exists $C > 0$ such that, for all $f \in \H_0^{\infty} (\S_{\theta})$, one has
\begin{align*}
 \| f(B) \|_{\Lop(X)} \leq C \sup_{z \in \S_{\theta}} \lvert f (z) \rvert
\end{align*}
is the question of the boundedness of the $\H^{\infty} (\S_{\theta})$-calculus of $B$ \cite[Prop.~5.3.4]{Haase}. While there is a vast and beautiful theory around this question, see, e.g., Haase~\cite{Haase} for further reading, we only want to note one consequence of the boundedness of the $\H^{\infty} (\S_{\theta})$-calculus of a sectorial operator~$B$. 
Indeed, there is a connection between the domains of fractional powers of $B$ and the complex interpolation spaces between $X$ and $\dom(B)$. 
More precisely, by~\cite[Thm.~6.6.9]{Haase} one finds with equivalent norms that
\begin{align}
\label{Eq: Complex interpolation}
 \dom(B^s) = \big[ X , \dom(B) \big]_s\,, \qquad s \in (0 , 1).
\end{align}

That, for $\theta \in (0 , \pi)$, the $\H^{\infty} (\S_{\theta})$-calculus of the Stokes operator on $\L^p_{\sigma} (\Omega)$ for 
\begin{align*}
\Big\lvert \frac{1}{p} - \frac{1}{2} \Big\rvert \leq \frac{1}{2d} + \varepsilon
\end{align*}
and $d \geq 3$ is indeed bounded, is a result of Kunstmann and Weis~\cite[Thm.~16]{Kunstmann_Weis}. In the following, we will shortly review their proof to confirm that their result stays valid in the two-dimensional case.

\subsection*{Step 1: the density of the domain and the range}

Concerning the density of the domain of the Stokes operator, we note the following lemma:

\begin{lemma}
\label{Lem: Embedding}
Let $\Omega \subset \IR^2$ be a bounded Lipschitz domain. Then there exists $\eps > 0$ such that, for all
\begin{align*}
 \Big\lvert \frac{1}{p} - \frac{1}{2} \Big\rvert < \frac{1}{4} + \eps ,
\end{align*}
the space $\W^{2 , p}_{0 , \sigma} (\Omega)$ is embedded continuously into $\dom(A_p)$. In particular, the representation formula
\begin{align}
\label{Eq: Stokes for smooth functions}
 A_p u = - \IP_p \Delta u, \qquad u \in \W^{2 , p}_{0 , \sigma} (\Omega),
\end{align}
is valid.
\end{lemma}

\begin{proof}
The proof of this in three or more dimensions is presented in~\cite[Lem.~2.5]{Tolksdorf} and is literally the same in two spatial dimensions. Thus, we may omit the proof.
\end{proof}

The density of the range of the Stokes operator follows by ``abstract nonsense''. Indeed, since $0 \in \rho(A_p)$, we know that $A_p$ is injective. Moreover, by the results from Subsection~\ref{Sec: Resolvent estimates and maximal regularity}, we know that $A_p$ is sectorial. The sectoriality combined with the reflexivity of $\L^p_{\sigma} (\Omega)$ implies the validity of the following algebraic and topological decomposition of $\L^p_{\sigma} (\Omega)$, see~\cite[Prop.~2.1.1.~h)]{Haase}:
\begin{align*}
 \L^p_{\sigma} (\Omega) = \ker(A_p) \oplus \overline{\Rg (A_p)}.
\end{align*}
Here, $\ker(A_p)$ denotes the kernel of $A_p$ and $\Rg(A_p)$ its range. Since $\ker(A_p) = \{ 0 \}$, we directly find that $\Rg(A_p)$ is dense in $\L^p_{\sigma} (\Omega)$.

\subsection*{Step 2: the comparison principle of Kunstmann and Weis}

The boundedness of the $\H^{\infty} (\S_{\theta})$-calculus of $A_p$ shall be deduced by that of the Dirichlet-Laplacian $- \Delta_p$ on $\L^p (\Omega ; \IC^2)$. That the $\H^{\infty} (\S_{\theta})$-calculus of $- \Delta_p$ is indeed bounded follows for example by combining the facts that the semigroup $(\e^{t \Delta_p})_{t \geq 0}$ satisfies heat kernel bounds, see, e.g., Davies~\cite[Cor.~3.2.8]{Davies}, and that $- \Delta_2$ has a bounded $\H^{\infty} (\S_{\theta})$-calculus on $\L^2 (\Omega ; \IC^2)$ (by the spectral theorem for self-adjoint operators) with a result of Duong and Robinson~\cite[Thm.~3.1]{Duong_Robinson}.

The \emph{comparison principle} of Kunstmann and Weis now reads as follows, see~\cite[Thm.~9]{Kunstmann_Weis}:
\begin{theorem}
\label{Thm: Kunstmann-Weis Hinfty}
Let $X$ and $Y$ be Banach spaces. Let $R \colon Y \to X$ and $S \colon X \to Y$ be bounded linear operators satisfying $R S = \Id_X$. 
Let $B$ have a bounded $\H^{\infty} (\S_{\sigma})$-calculus in $Y$ for some $\sigma \in (0,\pi)$, and let $A$ be $\cR$-sectorial in $X$. 
Assume that there are functions $\varphi , \psi \in \H_0^{\infty} (\S_{\nu}) \setminus \{ 0 \}$ where $\nu \in (0 , \sigma)$ and $C_1, C_2 > 0$ such that, for some $\beta > 0$ and all $\ell\in \IZ$,
\begin{align}
 \sup_{1 \leq s , t \leq 2} \cR \big\{ \varphi (s 2^{j + \ell} A) R \psi(t 2^j B) : j \in \IZ \big\} &\leq C_1 2^{- \beta \lvert \ell \rvert} \label{Eq: Condition} \quad\text{and}\\
 \sup_{1 \leq s , t \leq 2} \cR \big\{ \varphi (s 2^{j + \ell} A)^{\prime} S^{\prime} \psi (t 2^j B)^{\prime} : j \in \IZ \big\} &\leq C_2 2^{- \beta \lvert \ell \rvert}. \label{Eq: Dual condition}
\end{align}
Then $A$ has a bounded $\H^{\infty} (\S_{\nu})$-calculus on $X$.
\end{theorem}

\begin{remark}
The original result of Kunstmann and Weis is even stronger as they actually do not assume $A$ to be $\cR$-sectorial but only \emph{almost $\cR$-sectorial}. 
\end{remark}

Kunstmann and Weis provided in~\cite[Prop.~10 and Prop.~11]{Kunstmann_Weis} also tools for establishing~\eqref{Eq: Condition} and~\eqref{Eq: Dual condition} which are summarized below. 

\begin{proposition}
\label{Prop: Fractional comparison}
In the setting of Theorem~\ref{Thm: Kunstmann-Weis Hinfty}, suppose that there exist $\alpha_0 > 0$ and $C > 0$ such that,
for $\alpha = \pm \alpha_0$, we have
\begin{align}\label{Eq: Retraction Condition}
 R \dom(B^{\alpha}) \subset \dom(A^{\alpha}), \qquad \| A^{\alpha} R y \|_X \leq C \| B^{\alpha} y \|_Y \quad \text{for all} \quad y \in \dom(B^{\alpha}),
\end{align}
and
\begin{align}\label{Eq: Injection Condition}
 S \dom(A^{\alpha}) \subset \dom(B^{\alpha}), \qquad \| B^{\alpha} S x \|_Y \leq C \| A^{\alpha} x \|_X \quad \text{for all} \quad x \in \dom(A^{\alpha}).
\end{align}
Then Condition~\eqref{Eq: Condition} and Condition~\eqref{Eq: Dual condition} hold for the choice $C_1 = C_2 = C$, $\beta = \alpha_0$, and $\varphi(\lambda) = \psi(\lambda) = \lambda^{2\alpha_0}(1 + \lambda)^{-4 \alpha_0}$.
\end{proposition}

Let us assume for a moment that the assumptions of Proposition~\ref{Prop: Fractional comparison} are verified for the choice $X = \L^2_{\sigma} (\Omega)$, $Y = \L^2 (\Omega ; \IC^2)$, $R = \IP_2$, $S$ being the inclusion of $\L^2_{\sigma} (\Omega)$ into $\L^2 (\Omega ; \IC^2)$, $A = A_2$, and $B = - \Delta_2$. 
Let $\varphi, \psi \in \H_0^\infty(\S_\nu) \setminus \{0 \}$ denote the functions provided by Proposition~\ref{Prop: Fractional comparison}. Then we would find constants $C_1 , C_2 > 0$ and some $\beta > 0$ such that
\begin{align}
 \sup_{1 \leq s , t \leq 2} \cR \big\{ \varphi (s 2^{j + \ell} A_2) R \psi(- t 2^j \Delta_2) : j \in \IZ \big\} &\leq C_1 2^{- \beta \lvert \ell \rvert} \label{Eq: Condition applied} \quad\text{and}\\
 \sup_{1 \leq s , t \leq 2} \cR \big\{ \varphi (s 2^{j + \ell} A_2)^{\prime} S^{\prime} \psi (- t 2^j \Delta_2)^{\prime} : j \in \IZ \big\} &\leq C_2 2^{- \beta \lvert \ell \rvert} \label{Eq: Dual condition applied}
\end{align}
for all $\ell \in \IZ$.
Let further $p$ satisfy
\begin{align*}
 0 < \frac{1}{2} - \frac{1}{p} < \frac{1}{4} + \eps,
\end{align*}
where $\eps > 0$ is small enough such that $A_p$ is $\cR$-sectorial on $\L^p_{\sigma} (\Omega)$. Notice that also $- \Delta_p$ is $\cR$-sectorial on $\L^p (\Omega ; \IC^2)$ due to the Gaussian upper bounds of the heat semigroup, see~\cite[Thm.~3.1]{Hieber_Pruess}. Now, the $\cR$-sectoriality of these two operators together with~\cite[Lem.~3.3]{Kalton_Kunstmann_Weis} implies that the two sets
\begin{align*}
 \big\{ \varphi(s 2^{j + \ell} A_p) : s > 0 , j , \ell \in \IZ \big\} \subset \Lop(\L^p_{\sigma} (\Omega)) \quad \text{and} \quad \big\{ \psi(- t 2^{j} \Delta_p) : t > 0 , j \in \IZ \big\} \subset \Lop(\L^p (\Omega ; \IC^2))
\end{align*}
are $\cR$-bounded. Next, since singletons of bounded operators are always $\cR$-bounded and since products of $\cR$-bounded sets of operators are $\cR$-bounded as well, cf.~\cite[Prop.~3.4]{Denk_Hieber_Pruess}, we find that
\begin{align*}
 \big\{ \varphi(s 2^{j + \ell} A_p) R \psi(- t 2^j \Delta_p) : s , t > 0 , j , \ell \in \IZ \big\} \subset \Lop(\L^p (\Omega ; \IC^2) , \L^p_{\sigma} (\Omega))
\end{align*}
is $\cR$-bounded. Finally, since $\cR$-boundedness implies uniform boundedness, there exists $C > 0$ such that
\begin{align*}
 \sup_{\ell \in \IZ} \sup_{1 \leq s , t \leq 2} \cR\big\{ \varphi(s 2^{j + \ell} A_p) R \psi(- t 2^j \Delta_p) : j \in \IZ \big\} \leq C.
\end{align*}
Using the interpolation result in~\cite[Prop.~3.7]{Kalton_Kunstmann_Weis} together with~\eqref{Eq: Condition applied}, one finds a constant $C > 0$ such that, for $2 < q < p$ with
\begin{align*}
 \frac{1}{q} = \frac{1 - \theta}{2} + \frac{\theta}{p} \quad \text{for some} \quad 0 < \theta < 1,
\end{align*}
the following estimate holds for all $\ell \in \IZ$
\begin{align*}
\sup_{1 \leq s , t \leq 2} \cR\big\{ \varphi(s 2^{j + \ell} A_q) R \psi(- t 2^j \Delta_q) : j \in \IZ \big\} \leq C 2^{- (1 - \theta) \beta \lvert \ell \rvert}.
\end{align*}
Consequently, with the definitions $X = \L^q_{\sigma} (\Omega)$, $Y = \L^q (\Omega ; \IC^2)$, $R = \IP_q$, $S$ being the inclusion of $\L^q_{\sigma} (\Omega)$ into $\L^q (\Omega ; \IC^2)$, $A = A_q$, and $B = - \Delta_q$, Condition~\eqref{Eq: Condition} of Theorem~\ref{Thm: Kunstmann-Weis Hinfty} is satisfied.

Condition~\eqref{Eq: Dual condition} follows in a similar fashion by noticing that, due to self-adjointness, we have that
\begin{align*}
 \psi(- t 2^j \Delta_p)^{\prime} = \psi(- t 2^j \Delta_p^{\prime}) \simeq \psi(- t 2^j \Delta_{p^{\prime}})
\end{align*}
and
\begin{align*}
 \psi(t 2^j A_p)^{\prime} = \psi(t 2^j A_p^{\prime}) \simeq \psi(t 2^j A_{p^{\prime}}).
\end{align*}
Moreover, the dual operator $S^{\prime}$ of $S$ is identified with an operator on $\L^{p^{\prime}} (\Omega ; \IC^2)$ as follows. 
Let ${\bf f} \in \L^p (\Omega ; \IC^2)^{\prime}$, and let $f \in \L^{p^{\prime}} (\Omega ; \IC^2)$  denote its canonical identification. Then, for $g \in \L^{p^\prime}_\sigma(\Omega)$, one calculates
\begin{align*}
 \langle S^{\prime} {\bf f} , g \rangle_{(\L^p_{\sigma})^{\prime} , \L^p_{\sigma}} = \langle {\bf f} , S g \rangle_{(\L^p_{\sigma})^{\prime} , \L^p_{\sigma}} = \langle f , g \rangle_{\L^{p^{\prime}} , \L^p} = \langle f , \IP_p g \rangle_{\L^{p^{\prime}} , \L^p} = \langle \IP_{p^{\prime}} f , g \rangle_{\L^{p^{\prime}}_{\sigma} , \L^p_{\sigma}}.
\end{align*}
Consequently, the operator 
\begin{align*}
 \varphi (s 2^{j + \ell} A_p)^{\prime} S^{\prime} \psi (- t 2^j \Delta_p)^{\prime}
\end{align*}
may be identified with the operator
\begin{align*}
 \varphi (s 2^{j + \ell} A_{p^{\prime}}) \IP_{p^{\prime}} \psi (- t 2^j \Delta_{p^{\prime}}).
\end{align*}
Now, the same argument leading to~\eqref{Eq: Condition} can be used to establish~\eqref{Eq: Dual condition}. 
The only difference is to use the $\cR$-sectoriality of $A_{p^{\prime}}$ on $\L^{p^{\prime}}_{\sigma} (\Omega)$ and of $- \Delta_{p^{\prime}}$ on $\L^{p^{\prime}} (\Omega ; \IC^2)$. 
Hence, besides the verification of the conditions of Proposition~\ref{Prop: Fractional comparison} on the $\L^2$-scale, this establishes the boundedness of the $\H^{\infty}$-calculus of the Stokes operator on the $\L^p_{\sigma}$-scale and thus proves Theorem~\ref{Thm: Hinfty}.

\subsection*{Step 3: verification of the conditions from Proposition~\ref{Prop: Fractional comparison}}

To start with, we briefly introduce suitable scales of Bessel potential spaces, see~\cites{Jerison_Kenig_Dirichlet, Mitrea_Monniaux,Triebel_Functions, Triebel} for proofs of the stated results and for further information. For $s \in \IR$ and $1 < p < \infty$, the well-known \emph{Bessel potential space} on $\IR^2$ is defined by
\begin{align*}
    \H^{s,p}(\IR^2; \IC^2) 
    \coloneqq 
    \big\{ f\in \cS(\IR^2; \IC^2)^\prime : \cF^{-1} ( 1 + |\cdot|^2 )^{s/2} \cF(f)  \in \L^p(\IR^2; \IC^2) \big\}\,, 
\end{align*}
with the norm
\begin{align*}
    \| f \|_{ \H^{s,p}(\IR^2; \IC^2) } \coloneqq \| \cF^{-1} ( 1 + |\cdot |^2 )^{s/2} \cF(f) \|_{\L^p(\IR^2; \IC^2)}.
\end{align*}
Its counterpart on the domain $\Omega$ is defined via restriction
\begin{align*}
 \H^{s , p} (\Omega ; \IC^2) \coloneqq \big\{ \mathfrak{R}_\Omega (g) : g \in \H^{s,p}(\IR^2; \IC^2) \big\}\,,
\end{align*}
where $\mathfrak{R}_\Omega$ restricts distributions to $\Omega$ and the corresponding norm is given by the natural quotient norm
\begin{align*}
 \| f \|_{ \H^{s,p}(\Omega; \IC^2) } \coloneqq \inf_{\substack{g \in \H^{s,p}(\IR^2; \IC^2) \\  \mathfrak{R}_\Omega(g) = f}} \| g \|_{\H^{s,p}(\IR^2; \IC^2)}.
\end{align*}
To incorporate traces that vanish at the boundary of $\Omega$, one defines
\begin{align*}
    \H^{s,p}_0(\Omega; \IC^2) 
    \coloneqq 
    \Big\{ \mathfrak{R}_\Omega (g) : g \in \H^{s,p}(\IR^2; \IC^2), \, \supp g \subset \overline{\Omega} \Big\}
\end{align*}
with the quotient norm
\begin{align}\label{Eq: Quotient Norm}
    \| f \|_{ \H^{s,p}_0(\Omega; \IC^2) } \coloneqq \inf_{\substack{g \in \H^{s,p}(\IR^2; \IC^2) \\ \supp g \subset \overline{\Omega}}} \| g \|_{\H^{s,p}(\IR^2; \IC^2)}.
\end{align}
The spaces $\H^{s , p}_0 (\Omega ; \IC^2)$ and $\H^{s , p} (\Omega ; \IC^2)$ coincide if $-1 + 1/p < s < 1/p$. Thus, if this condition applies, we may also write $\H^{s , p}_0 (\Omega ; \IC^2)$ for $\H^{s , p} (\Omega ; \IC^2)$ if this simplifies the notation. Moreover, $\C_c^{\infty} (\Omega ; \IC^2)$ is dense in $\H^{s , p}_0 (\Omega ; \IC^2)$ for all $s \in \IR$ and $1 < p < \infty$. Finally, for $s > 0$, the space $\H^{s , p}_0 (\Omega ; \IC^2)$ is reflexive. In particular, it holds for $1/p + 1/p^{\prime} = 1$ that
\begin{align*}
 \H^{s , p}_0 (\Omega ; \IC^2)^{\prime} = \H^{- s , p^{\prime}} (\Omega ; \IC^2) \quad \text{and} \quad \H^{- s , p^{\prime}} (\Omega ; \IC^2)^{\prime} = \H^{s , p}_0 (\Omega ; \IC^2).
\end{align*}
If we consider Bessel potential spaces as subspaces of  scalar-valued tempered distributions $\cS(\IR^2)^\prime$, we just write $\H^{s,p}(\IR^2)$, $\H^{s,p}(\Omega)$, and $\H_0^{s,p}(\Omega)$.

The solenoidal counterparts of these spaces are defined for $s > -1 + 1/p$ as the~\emph{Stokes scale} associated to the Lipschitz domain $\Omega$ and are given by
\begin{align*}
    \H^{s,p}_{0,\sigma}(\Omega; \IC^2) \coloneqq \overline{\C_{c,\sigma}^\infty(\Omega)}^{\|\cdot\|_{\H_0^{s,p}(\Omega; \IC^2)}}.
\end{align*}
We remark that this definition is in line with the scale used by Mitrea and Monniaux due to their result in~\cite[Prop.~2.10]{Mitrea_Monniaux}.

We will need the following result about the Helmholtz projection on Bessel potential spaces~\cite[Prop.~2.16~(II)]{Mitrea_Monniaux}.

\begin{proposition}[Mitrea, Monniaux]\label{prop: Helmholtz on Bessel potential space}
For $|s| < \frac{1}{2}$, the Helmholtz projection $\IP$ acts as a bounded linear projection on $\H^{s,2}(\Omega; \IC^2)$ and yields the following topological direct sum:
\begin{align*}
    \H^{s,2}(\Omega; \IC^2) = \H^{s,2}_{0,\sigma}(\Omega) \oplus \nabla \H^{s + 1,2}(\Omega)\;.
\end{align*}
Furthermore, $\H^{s , 2}_{0,\sigma} (\Omega)$ is the range of $\IP$ and is reflexive with $\H^{s , 2}_{0,\sigma} (\Omega)^\prime = \H^{- s , 2}_{0,\sigma} (\Omega)$.
\end{proposition}

In the next step, we establish a relation between the scales of Bessel potential spaces from Proposition~\ref{prop: Helmholtz on Bessel potential space} and suitable interpolation and extrapolation scales of Banach spaces that are induced by fractional powers of $A_2$ and $-\Delta_2$, see~\cite[Def.~15.21]{Kunstmann_Weis_2004}.

For a given Banach space $X$, a sectorial operator $A \colon \cD(A) \to X$ with $0 \in \rho(A)$, and $\alpha \in \IR$, we define
\begin{align*}
    \dot{X}_{\alpha, A} \coloneqq ( \cD(A^\alpha), \|A^\alpha \cdot\|_X )^{\sim}
\end{align*}
to be the completion of the domain $\cD(A^\alpha)$ with respect to the \emph{homogeneous graph norm}.

\begin{remark}\label{Rem: interpolation scale}
\begin{enumerate}[label=\normalfont{(\roman*)}]
\item\label{Item: Dual interpolation scale} If $X$ is reflexive, there is a natural isomorphism $\bigl(\dot{X}_{\alpha, A}\bigr)^{\prime} = \bigl(X^{\prime}\bigr)^{\textbf{.}}_{-\alpha, A^{\prime}}$, see~\cite[Prop.~15.23]{Kunstmann_Weis_2004}. 
In particular, if $X$ is a Hilbert space and $A$ is self-adjoint, then $\bigl(\dot{X}_{\alpha, A}\bigr)^{\prime} = \dot{X}_{-\alpha, A}$ via the usual identification of $A$ with $A^{\prime}$ via the Riesz isomorphism.
\item\label{Item: Completeness for s > 0} For sectorial operators with $0 \in \rho(A)$, the scale $\bigl( \dot{X}_{\alpha, A}\bigr)$ coincides with the usual \emph{extrapolated fractional power scale of order 1} or the \emph{interpolation-extrapolation scale} according to \cite[Thm.~V.1.3.8 and Thm.~V.1.5.4]{Amann-monograph}.
In particular, for $\alpha > 0$, the domain $\cD(A^\alpha)$ is already complete with respect to the homogeneous graph norm as a consequence of the closed graph theorem.
\end{enumerate}
\end{remark}

The following proposition, whose first part is due to Mitrea and Monniaux~\cite[Thm.~5.1]{Mitrea_Monniaux}, characterizes the fractional power domains of the Stokes operator~$A_2$ and the Dirichlet-Laplacian~$-\Delta_2$ on $\L^2(\Omega; \IC^2)$ in terms of the Bessel potential spaces from Proposition~\ref{prop: Helmholtz on Bessel potential space}.

\begin{proposition}\label{Prop: Fractional Power Domains Stokes and Laplace}
    Let $|s| < \frac{1}{2}$.
    Then
    \begin{align}\label{eq: Bessel potential domains}
         \bigl(\L^2_\sigma(\Omega)\bigr)^{\textbf{.}}_{s/2, A_2} = \H^{s,2}_{0,\sigma}(\Omega)
         \quad\text{and}\quad
         \bigl(\L^2(\Omega; \IC^2)\bigr)^{\textbf{.}}_{s/2, -\Delta_2} = \H_0^{s,2}(\Omega; \IC^2).
    \end{align}
    In particular, for $s > 0$,
    \begin{align}\label{eq: Bessel potential domains positive s}
         \cD(A_2^{s/2}) = \H^{s,2}_{0,\sigma}(\Omega)
         \quad\text{and}\quad
         \cD((-\Delta_2)^{s/2}) = \H_0^{s,2}(\Omega; \IC^2).
    \end{align}
\end{proposition}
\begin{proof}
    We only prove the facts for the Stokes operator.
    The corresponding facts for the Dirichlet-Laplacian follow for example from~\cite[Sect.~7]{Jerison_Kenig}.
    
    We start with the case $s > 0$. By virtue of the invertibility of $A_2$ and of Remark~\ref{Rem: interpolation scale}~\ref{Item: Completeness for s > 0}, we only have to calculate $\cD(A_2^{s/2})$. Since $A_2$ is self-adjoint, its $\H^{\infty}$-calculus on $\L^2_{\sigma} (\Omega)$ is bounded, see, e.g., \cite[Sect.~2.4]{Denk_Hieber_Pruess}. Thus, employing~\eqref{Eq: Complex interpolation} one finds that $\cD(A_2^{s/2}) = \big[ \L^2_\sigma(\Omega), \cD(A_2^{1/2})\big]_s$.
    Furthermore, it is known that $\cD(A_2^{1/2}) = \W^{1,2}_{0,\sigma}(\Omega)$, see, e.g., ~\cite[Lem.~III.2.2.1]{Sohr}
    and that the arising interpolation space is computed as 
    \begin{align*}
        \big[ \L^2_\sigma(\Omega), \W^{1,2}_{0 , \sigma}(\Omega)\big]_s = \H_{0,\sigma}^{s,2}(\Omega), 
    \end{align*}
    see~\cite[Thm.~2.12]{Mitrea_Monniaux}. 
    
    Now, let $s < 0$. 
    Using Proposition~\ref{prop: Helmholtz on Bessel potential space}, the fact that $A_2$ is self-adjoint, the isomorphism from Remark~\ref{Rem: interpolation scale}~\ref{Item: Dual interpolation scale}, and the result for the case $s > 0$ yield
    \begin{align*}
        \H_{0,\sigma}^{s,2}(\Omega) 
        =  \H_{0,\sigma}^{-s,2}(\Omega)^{\prime} 
        = \bigl(\big(\L^2_\sigma(\Omega)\big)^{\textbf{.}}_{-s/2, A_2} \bigr)^{\prime}
        = \big(\L^2_\sigma(\Omega)\big)^{\textbf{.}}_{s/2, A_2} 
    \end{align*}
    which completes the proof of the statement.
\end{proof}

For $B_2 \coloneqq -\Delta_2$ and $s > 0$, the following diagram summarizes the interplay of Proposition~\ref{prop: Helmholtz on Bessel potential space} (vertical arrows) and Proposition~\ref{Prop: Fractional Power Domains Stokes and Laplace} (horizontal arrows).
\begin{center}
\begin{tikzcd}
\H_0^{s,2} \arrow[d, "\IP"] \arrow[r, leftrightarrow, "\simeq"] & \cD(B_2^{s/2}) \arrow[rr, "B_2^{s/2}"]  && \L^2 \arrow[d, "\IP_2"]  \\
\H^{s,2}_{0,\sigma} \arrow[r, leftrightarrow, "\simeq"]           & \cD(A_2^{s/2}) \arrow[rr, "A_2^{s/2}"]  && \L^2_\sigma
\end{tikzcd}
\end{center}

We have now gathered all the prerequisites needed to verify the conditions from Proposition~\ref{Prop: Fractional comparison}.
As in Step~2, we let $X = \L^2_\sigma(\Omega)$, $Y = \L^2(\Omega; \IC^2)$, $R = \IP_2$, and use for $S \colon X \to Y$ the inclusion map. 
Furthermore, we fix some $0 < \alpha_0 < 1/4$, and we carry out the proof in two separate cases depending on the sign of the parameter $\alpha$.

\subsubsection*{The case $\alpha = \alpha_0 > 0$}
In this case, 
\begin{align*}
    R \cD(B_2^\alpha) = R \H^{2\alpha,2}_0(\Omega ; \IC^2) = \H^{2\alpha,2}_{0,\sigma}(\Omega) = \cD(A_2^\alpha)
\end{align*}
by~\eqref{eq: Bessel potential domains positive s}, Proposition~\ref{prop: Helmholtz on Bessel potential space}, and the characterization of the fractional power domains of the Dirichlet-Laplacian. Moreover, for all $y \in \H^{2\alpha,2}_0(\Omega ; \IC^2)$, one estimates
\begin{align}\label{Eq: Retraction alpha gt 0}
   \|A_2^\alpha R y \|_{\L^2_\sigma(\Omega)} 
   \lesssim \| R y \|_{\H^{2\alpha, 2}_{0,\sigma}(\Omega)}
   \lesssim \| y \|_{\H^{2\alpha, 2}(\Omega; \IC^2)}
   \lesssim \| B_2^{\alpha} y \|_{\L^2(\Omega; \IC^2)}\,.
\end{align}
This gives Condition~\eqref{Eq: Retraction Condition} for $\alpha > 0$.
Similarly, we verify that 
\begin{align*}
   S \cD(A_2^\alpha) 
   = \H^{2\alpha, 2}_{0,\sigma}(\Omega) 
   \subset \H^{2\alpha, 2}_0(\Omega; \IC^2) = \cD(B_2^\alpha)
\end{align*} and calculate for every $x \in \H^{2\alpha, 2}_{0,\sigma}(\Omega)$
\begin{align}\label{Eq: Injection alpha gt 0}
    \| B_2^\alpha S x \|_{\L^2(\Omega; \IC^2)}
    = \| B_2^\alpha x \|_{\L^2(\Omega; \IC^2)}
    \lesssim \| x \|_{\H^{2\alpha, 2}_{0,\sigma}(\Omega)}\,.
\end{align}
This gives Condition~\eqref{Eq: Injection Condition} for $\alpha > 0$.

\subsubsection*{The case $\alpha = - \alpha_0 < 0$}
For the case of negative exponents, the inclusions on the left-hand side of~\eqref{Eq: Retraction Condition} and \eqref{Eq: Injection Condition} are straightforward since $\cD(A_2^\alpha) = \L^2_\sigma(\Omega)$ and $\cD(B_2^\alpha) = \L^2(\Omega; \IC^2)$ as sets. 
The first inclusion follows from the classical mapping properties of the Helmholtz projection on $\L^2$, see~\cite[Lem.~2.5.2]{Sohr}, while the second one is trivial.

The boundedness estimate in Condition~\eqref{Eq: Retraction Condition} follows via duality from~\eqref{Eq: Injection alpha gt 0}. Indeed, let $y \in \L^2(\Omega; \IC^2)$ and $g \in \L^2_\sigma(\Omega)$. 
Then
\begin{align*}
    \langle A_2^\alpha R B_2^{-\alpha} y, g \rangle_{\L^2_\sigma, \L^2_\sigma}^{}
    = \langle y, B_2^{-\alpha} S A_2^\alpha g \rangle_{\L^2, \L^2}^{},
\end{align*}
and the claim follows by taking the supremum over all $g$.
For the remaining part of Condition~\eqref{Eq: Injection Condition}, note that, since $\alpha < 0$, we have $A_2^{-\alpha} x \in \cD(A_2^\alpha) = \L^{2}_\sigma(\Omega)$ implying the identity $S A_2^{-\alpha} = R A_2^{-\alpha}$.
Now, the desired estimate follows via duality from~\eqref{Eq: Retraction alpha gt 0} as, for $x \in \L^2_\sigma(\Omega)$ and $h \in \L^2(\Omega; \IC^2)$, we have
\begin{align*}
    \langle B_2^\alpha S A_2^{-\alpha} x, h \rangle_{\L^2, \L^2}^{}
    = \langle x, A_2^{-\alpha} R B_2^{\alpha} h \rangle_{\L^2_\sigma, \L^2_\sigma}^{}\,.
\end{align*}

\subsection{Domains of fractional powers}

This section deals with the calculation of domains of fractional powers for the Stokes operator on $\L^p_{\sigma} (\Omega)$. This extends the results from the case $p = 2$ established in~\cite[Thm.~5.1]{Mitrea_Monniaux}. The same approach could also be used to extend the results in three and higher dimensions in the $\L^p_{\sigma}$-situation, where currently only the domains of $A_p^{\theta}$ are characterized for $0 \leq \theta \leq 1 / 2$. We will, however, concentrate on the two-dimensional case. 
As a preparation, we state a regularity result for the Poisson problem for the Stokes system with Dirichlet boundary conditions, initially formulated by Dindo\v{s} and Mitrea \cite[Thm.~5.6]{Dindos_Mitrea} and later improved by Mitrea and Wright~\cite[Thm.~10.6.2]{Mitrea_Wright}.  
We remark that this theorem is formulated in terms of Besov spaces $\B^s_{p , q}$ and Triebel--Lizorkin spaces $\F^s_{p , q}$ and that we present the particular case of Bessel potential spaces satisfying the relation $\H^{s , p} = \F^s_{p , 2}$.

\begin{theorem}[Mitrea, Wright]
\label{Thm: Regularity Stokes}
Let $\Omega \subset \IR^2$ be a bounded Lipschitz domain. Then there exists $0 < \delta \leq 1$ depending only on $\Omega$ such that, for all $1 < p < \infty$ and $0 < s < 1$ satisfying either
\begin{align}
\label{Eq: First condition}
 0 < \frac{1}{p} < s + \frac{1 + \delta}{2} \quad \text{and} \quad 0 < s \leq \frac{1 + \delta}{2}
\end{align}
or
\begin{align}
\label{Eq: Second condition}
 - \frac{1 + \delta}{2} < \frac{1}{p} - s < \frac{1 + \delta}{2} \quad \text{and} \quad \frac{1 + \delta}{2} < s < 1\,,
\end{align}
the Stokes system
\begin{align*}
\left\{ \begin{aligned}
 - \Delta u + \nabla \phi &= f && \text{in } \Omega, \\
 \divergence(u) &= 0 && \text{in } \Omega, \\
 u &= 0 && \text{on } \partial \Omega
 \end{aligned} \right.
\end{align*}
has for all $f \in \H^{s + 1 / p - 2 , p} (\Omega ; \IC^2)$ unique solutions $u \in \H^{s + 1 / p , p} (\Omega ; \IC^2)$ and $\phi \in \H^{s + 1 / p - 1 , p} (\Omega)$ (the pressure is unique up to the addition of constants). Moreover, there exists a constant $C > 0$ depending only on $p$, $s$, and $\Omega$ such that the following estimate holds
\begin{align*}
 \| u \|_{\H^{s + \frac{1}{p} , p} (\Omega ; \IC^2)} + \| \phi \|_{\H^{s + \frac{1}{p} - 1 , p} (\Omega)} \leq C \| f \|_{\H^{s + \frac{1}{p} - 2 , p} (\Omega ; \IC^2)}.
\end{align*}
\end{theorem}

\begin{remark}
\label{Rem: Domain embedding}
Let $1 < p < \infty$ satisfy~\eqref{Eq: Helmholtz two dimensions}. If $f \in \L^p_{\sigma} (\Omega)$, then $f \in \H^{s + 1 / p - 2 , p} (\Omega ; \IC^2)$ for all $s \in (0 , 1)$. Thus, 
\begin{align*}
 \dom(A_p) \subset \bigcap_s \, \H^{s + \frac{1}{p} , p} (\Omega ; \IC^2),
\end{align*}
where the intersection is taken over all $s \in (0 , 1)$ that either satisfy~\eqref{Eq: First condition} or~\eqref{Eq: Second condition}. If $\delta = 1$, then~\eqref{Eq: Second condition} is void what implies that
\begin{align*}
 \dom(A_p) \subset \bigcap_{t < 1 + \frac{1}{p}} \, \H^{t , p} (\Omega ; \IC^2).
\end{align*}
If $\delta \in (0 , 1)$, the first inequality in~\eqref{Eq: Second condition} implies that $s$ must satisfy $s < \min(1 , 1 / p + (1 + \delta) / 2)$. A calculation of the minimum reveals that
\begin{align*}
 \dom(A_p) \subset \bigcap_{t < 1 + \frac{1}{p}} \, \H^{t , p} (\Omega ; \IC^2) \quad \text{if} \quad \frac{1}{2} - \frac{1}{p} \leq \frac{\delta}{2}
\end{align*}
and that
\begin{align*}
 \dom(A_p) \subset \bigcap_{t < \frac{2}{p} + \frac{1 + \delta}{2}} \, \H^{t , p} (\Omega ; \IC^2) \quad \text{if} \quad \frac{1}{2} - \frac{1}{p} > \frac{\delta}{2}.
\end{align*}
\end{remark}

We are ready to prove the embedding of the Bessel potential spaces into domains of fractional powers $\cD(A_p^\theta)$ from Theorem~\ref{Thm: Fractional powers}.
\begin{theorem}\label{Thm: Continuous Embedding}
Let $\Omega \subset \IR^2$ be a bounded Lipschitz domain. Then there exists $\eps > 0$ such that, for all $1 < p < \infty$ satisfying~\eqref{Eq: Helmholtz two dimensions} and all $0 < \theta < 1$, the continuous embedding
\begin{align*}
 \H^{2 \theta , p}_{0 , \sigma} (\Omega) \subset \dom(A_p^{\theta})
\end{align*}
holds. Furthermore, there exists $\delta \in (0 , 1]$ such that, if $\theta$ and $p$ additionally satisfy either
\begin{align}
\label{Eq: Fractional small p}
 \theta < \frac{1}{2} + \frac{1}{2 p} \quad \text{if} \quad \frac{1}{2} - \frac{1}{p} \leq \frac{\delta}{2}
\end{align}
or
\begin{align}
\label{Eq: Fractional large p}
 \theta < \frac{1}{p} + \frac{1 + \delta}{4} \quad \text{if} \quad \frac{1}{2} - \frac{1}{p} > \frac{\delta}{2}\,,
\end{align}
we have with equivalent norms that
\begin{align*}
\H^{2 \theta , p}_{0 , \sigma} (\Omega)  =\dom(A_p^{\theta}) .
\end{align*}
\end{theorem}

\begin{proof}
The boundedness of the $\H^{\infty}$-calculus of $A_p$ by Theorem~\ref{Thm: Hinfty} and its consequence for complex interpolation \eqref{Eq: Complex interpolation} imply that
\begin{align}
\label{Eq: Fractional powers and interpolation}
 \big[ \L^p_{\sigma} (\Omega) , \dom(A_p^{\alpha}) \big]_{\theta} = \dom(A_p^{\alpha \theta})
\end{align}
for $\alpha > 0$.
Moreover, we find by Lemma~\ref{Lem: Embedding} the continuous embedding
\begin{align}\label{Eq: Embedding of interpolation}
 \big[ \L^p_{\sigma} (\Omega) , \W^{2 , p}_{0 , \sigma} (\Omega) \big]_{\theta} \subset \big[ \L^p_{\sigma} (\Omega) , \dom(A_p) \big]_{\theta}.
\end{align}
Due to the interpolation result in~\cite[Thm.~2.12]{Mitrea_Monniaux}, it is known that the interpolation space on the left-hand side of~\eqref{Eq: Embedding of interpolation} coincides with $\H^{2 \theta , p}_{0 , \sigma} (\Omega)$. 
Taking $\alpha = 1$ in~\eqref{Eq: Fractional powers and interpolation} results in the continuous embedding, which is formulated with $r$ instead of $p$ and $s$ instead of $\theta$ for further reference: for any $1 < r < \infty$ satisfying~\eqref{Eq: Helmholtz two dimensions} and any $s \in (0 , 1)$ it holds that
\begin{align}
\label{Eq: Continuous embedding}
 \H^{2 s , r}_{0 , \sigma} (\Omega) \subset \dom(A_r^s).
\end{align}

For the converse inclusion, let $\delta \in (0 , 1]$ be given by Theorem~\ref{Thm: Regularity Stokes}. We assume first that 
\begin{align}
\label{Eq: Auxiliary condition theta}
 \frac{1 - \delta}{4} + \frac{1}{2 p} < \theta < \min\bigg( \frac{1}{2} + \frac{1}{2 p} ,\; \frac{1}{p} + \frac{1 + \delta}{4} \bigg).
\end{align}
Notice that the minimum is given by $1 / 2 + 1 / (2 p)$ if $1 / 2 - 1 / p \leq \delta / 2$ and that the minimum is given by $1 / p + (1 + \delta) / 4$ if $1 / 2 - 1 / p > \delta / 2$. We remark that the desired embedding is established once the operator $A_p^{- \theta}$ is a bounded operator
\begin{align*}
 A_p^{- \theta} \colon \L^p_{\sigma} (\Omega) \to \H^{2 \theta , p}_{0 , \sigma} (\Omega).
\end{align*}
Since $\H^{2 \theta , p}_{0 , \sigma} (\Omega) = \H^{2 \theta , p}_0 (\Omega ; \IC^2) \cap \L^p_{\sigma} (\Omega)$ holds due to~\cite[Cor.~2.11]{Mitrea_Monniaux} and since $A_p^{- \theta}$ is bounded from $\L^p_{\sigma} (\Omega)$ to $\L^p_{\sigma} (\Omega)$, it suffices to prove the boundedness of
\begin{align}
\label{Eq: Boundedness property}
 A_p^{- \theta} \colon \L^p_{\sigma} (\Omega) \to \H^{2 \theta , p}_0 (\Omega ; \IC^2).
\end{align}
To this end, we find by the self-adjointness of $A_2$ and of $\IP_2$ for $f \in \C_{c , \sigma}^{\infty} (\Omega)$ and $g \in \C_c^{\infty} (\Omega ; \IC^2)$ that
\begin{align}
\label{Eq: Duality argument}
 \Big\lvert \int_{\Omega} A_p^{- \theta} f \cdot \overline{g} \; \d x \Big\rvert = \Big\lvert \int_{\Omega} f \cdot \overline{A_{p^{\prime}}^{- \theta} \IP_{p^{\prime}} g} \; \d x \Big\rvert \leq \| f \|_{\L^p_{\sigma}(\Omega)} \| A_{p^{\prime}}^{- \theta} \IP_{p^{\prime}} g \|_{\L^{p^{\prime}}(\Omega; \IC^2)}.
\end{align}
Since $\H^{- 2 \theta , p^{\prime}} (\Omega ; \IC^2)^{\prime} = \H^{2 \theta , p}_0 (\Omega ; \IC^2)$ and since $\C^{\infty}_c (\Omega ; \IC^2)$ is dense in $\H^{- 2 \theta , p^{\prime}} (\Omega ; \IC^2)$, cf.~\cite[Thm.~3.5~(i)]{Triebel}, the boundedness property~\eqref{Eq: Boundedness property} follows from~\eqref{Eq: Duality argument} once it is shown that there exists $C > 0$ such that, for all $g \in \C_c^{\infty} (\Omega ; \IC^2)$, it holds
\begin{align}\label{eq: Boundedness from Lp to H-2theta}
 \| A_{p^{\prime}}^{- \theta} \IP_{p^{\prime}} g \|_{\L^{p^{\prime}}(\Omega; \IC^2)} \leq C \| g \|_{\H^{- 2 \theta , p^{\prime}}(\Omega; \IC^2)}.
\end{align}

To this end, we find by virtue of the first inequality in~\eqref{Eq: Auxiliary condition theta} combined with Remark~\ref{Rem: Domain embedding} that $\Rg(A_{p^{\prime}}^{-1}) = \cD (A_{p^{\prime}}) \subset \H^{2 (1 - \theta) , p^{\prime}}_{0 , \sigma} (\Omega)$. This enables us to use~\eqref{Eq: Continuous embedding} with $r = p^{\prime}$ and $s = 1 - \theta$ to deduce that
\begin{align*}
 \| A_{p^{\prime}}^{- \theta} \IP_{p^{\prime}} g \|_{\L^{p^{\prime}}(\Omega; \IC^2)} 
 = \| A_{p^{\prime}}^{1 - \theta} A_{p^{\prime}}^{- 1} \IP_{p^{\prime}} g \|_{\L^{p^{\prime}}(\Omega; \IC^2)} 
 \leq C \| A_{p^{\prime}}^{- 1} \IP_{p^{\prime}} g \|_{\H^{2 - 2 \theta , p^{\prime}}_{0 , \sigma}(\Omega; \IC^2)}.
\end{align*}
Define $u \coloneqq A_{p^{\prime}}^{-1} \IP_{p^{\prime}} g$ and let $\phi$ denote the associated pressure. Then $u$ and $\phi$ solve the Stokes system given by
\begin{align*}
 \left\{ \begin{aligned}
 - \Delta u + \nabla \phi &= \IP_{p^{\prime}} g && \text{in } \Omega, \\
 \divergence(u) &= 0 && \text{in } \Omega, \\
 u &= 0 && \text{on } \partial \Omega.
 \end{aligned} \right.
\end{align*}
Now, by virtue of~\cite[Thm.~4.4]{MitreaD}, the function $\IP_{p^{\prime}} g$ is given by $g + \nabla h$ for some $h \in \W^{1 , p^{\prime}} (\Omega)$. It follows that $u$ and $\phi - h$ solve the same system but with the right-hand side $g$. Now, define $s \coloneqq 1 + 1 / p - 2 \theta$. Then the conditions imposed on $\theta$ in~\eqref{Eq: Auxiliary condition theta} imply that $s$ and $p^{\prime}$ satisfy~\eqref{Eq: First condition} so that Theorem~\ref{Thm: Regularity Stokes} implies that
\begin{align*}
 \| u \|_{\H^{2 - 2 \theta , p^{\prime}}_{0 , \sigma}(\Omega)} \leq C \| g \|_{\H^{- 2 \theta , p^{\prime}}(\Omega; \IC^2)}
\end{align*}
which in turn shows that~\eqref{eq: Boundedness from Lp to H-2theta} holds.
This concludes the proof in the case where $\theta$ satisfies the lower bound in~\eqref{Eq: Auxiliary condition theta}.

To get rid of the lower bound in~\eqref{Eq: Auxiliary condition theta}, choose, for general $\theta$ subject to~\eqref{Eq: Fractional small p} or~\eqref{Eq: Fractional large p}, some $\alpha > 0$ subject to the upper and lower bounds in~\eqref{Eq: Auxiliary condition theta} satisfying $0 < \theta < \alpha$. Then the result follows from the previous considerations combined with the interpolation result in~\cite[Thm.~2.12]{Mitrea_Monniaux} and~\eqref{Eq: Fractional powers and interpolation}.
\end{proof}

\subsection{The weak Stokes operator}\label{sec: weak Stokes}

In this section, we expand our functional analytic framework by a Stokes-like operator on spaces of regular distributions.
This operator will prove to be helpful in Section~\ref{Sec: Navier-Stokes} for establishing higher regularity of solutions to the Navier--Stokes equations with the right-hand side in divergence form, cf.~\cite[Sect.~7]{Choudhury_Hussein_Tolksdorf}.

For the course of this section, let $\Omega \subset \IR^2$ be a bounded Lipschitz domain and $\eps > 0$ such that, for all numbers $p$ that satisfy~\eqref{Eq: Helmholtz two dimensions}, the operator $- A_p$ generates a bounded analytic semigroup.
Moreover, for $1/p + 1/p^\prime = 1$, let 
$\Phi \colon \big[\L^{p^{\prime}}_{\sigma}(\Omega)\big]^{*} \rightarrow \L^p_{\sigma}(\Omega)$
denote the canonical isomorphism between the \emph{anti}dual $\big[\L^{p^{\prime}}_{\sigma}(\Omega)\big]^{*}$ and $\L^p_{\sigma}(\Omega)$ 
with the duality pairing 
\begin{align*}
\langle \Phi^{-1}u,v \rangle_{\big[\L^{p^{\prime}}_{\sigma}\big]^{*}, \L^{p^{\prime}}_{\sigma}} = \langle u,v \rangle_{\L^p_{\sigma}, \L^{p^{\prime}}_{\sigma}}= \int_ {\Omega} u \cdot \overline{v} \; \d x, \quad u \in  \L^p_{\sigma}(\Omega), \, v\in \L^{p^{\prime}}_{\sigma}(\Omega).
\end{align*} 
We regard $\Phi^{-1}$ also  as the canonical inclusion of $\L^p_{\sigma} (\Omega)$ into $\W^{-1 , p}_{\sigma} (\Omega)$ via 
\begin{align*}
\langle \Phi^{-1}u, v \rangle_{\W^{- 1 , p}_{\sigma},\W^{1,p^{\prime}}_{0,\sigma}}= \langle u,v \rangle_{\L^p_{\sigma}, \L^{p^{\prime}}_{\sigma}}, \quad u\in \L^p_{\sigma} (\Omega), \, v\in \W^{1,p^{\prime}}_{0,\sigma}(\Omega).
\end{align*} 
In this sense, we define the \emph{weak Stokes operator} $\mathcal{A}_p$ in $\W^{-1 , p}_{\sigma} (\Omega)$
by $\dom(\mathcal{A}_p) \coloneqq \Phi^{-1} \W^{1 , p}_{0 , \sigma} (\Omega)$ and 
\begin{equation*}
\mathcal{A}_p \colon \dom(\mathcal{A}_p) \subset \W^{-1 , p}_{\sigma} (\Omega) \rightarrow \W^{-1,p}_{\sigma}(\Omega), \quad w \mapsto \Big[v \mapsto \int_{\Omega} \nabla \Phi w \cdot \overline{ \nabla v } \; \d x \Big].
\end{equation*}
Recall that, by Theorem~\ref{Thm: Continuous Embedding} and the invertibility of $A_p^\prime$, cf.~Theorem~\ref{Thm: Resolvent} and Remark~\ref{Rem: R-boundedness}~\ref{it: Strongly Continuous SG}, the square root of the Stokes operator is an isomorphism
$A^{\frac{1}{2}}_{p^{\prime}} \colon \W^{1,p^{\prime}}_{0,\sigma}(\Omega) \to \L^{p^{\prime}}_{\sigma}(\Omega)$.
This fact allows to deduce the relation
\begin{equation*}
 \big[A^{\frac{1}{2}}_{p^{\prime}} \big]^{*}\Phi^{-1}  \in \text{Isom}(\L^{p}_{\sigma}(\Omega),\W^{-1,p}_{\sigma}(\Omega))
\end{equation*}
and use the operator~$\big[A^{\frac{1}{2}}_{p^{\prime}} \big]^{*}\Phi^{-1}$ as a similarity transform that connects $A_p$ and $\cA_p$, namely
\begin{equation}\label{Eq: representation weak Stokes}
    \cA_{p}= \big[A^{\frac{1}{2}}_{p'} \big]^* \Phi^{-1} \circ A_{p} \circ A^{-\frac{1}{2}}_{p} \Phi = \big[A^{\frac{1}{2}}_{p'} \big]^* \Phi^{-1} \circ A_{p} \circ \Phi \big[A^{-\frac{1}{2}}_{p'} \big]^*.
\end{equation}
The representation formulas~\eqref{Eq: representation weak Stokes} were derived in~\cite[Lem.~5.1]{Choudhury_Hussein_Tolksdorf} for the case $d = 3$. 
However, the presented proof carries over \emph{mutatis mutandis} to the case $d = 2$.

Permanence properties of the class of $\cR$-sectorial operators dictate that $\cA_p$ inherits the $\cR$-sectoriality of $A_p$, see, e.g., \cite[Sect.~4.1]{Denk_Hieber_Pruess}.
In particular, the representation~\eqref{Eq: representation weak Stokes} induces a representation of the resolvents
\begin{equation}\label{Eq: Resolvent representation}
    ( \lambda + \cA_p)^{-1} = \big[A^{\frac{1}{2}}_{p'} \big]^* \Phi^{-1} \circ (\lambda +  A_{p})^{-1} \circ \Phi \big[A^{-\frac{1}{2}}_{p'} \big]^* \,,
\end{equation}
see the proof of \cite[Prop.~1.3~(iv)]{Denk_Hieber_Pruess}.
The following result is a corollary of~\eqref{Eq: Resolvent representation} in terms of a similar semigroup and its implication on maximal regularity. 
\begin{proposition}
\label{Prop: Analyticity of weak Stokes}
Let $\Omega \subset \IR^2$ be a bounded Lipschitz domain. Then there exists $\eps > 0$ such that, for all numbers $p$ that satisfy~\eqref{Eq: Helmholtz two dimensions}, the following statements are valid.
\begin{enumerate}[label=\normalfont{(\roman*)}]
\item\label{it: Analytic weak stokes}  $\rho(\cA_p) = \rho(A_p)$ and 
$-\cA_p$ generates a bounded analytic semigroup $\e^{-t\cA_p}$ on $\W^{-1 , p}_{\sigma} (\Omega)$.
\item\label{it: consistency} For $u \in \W^{-1 , p}_{\sigma} (\Omega)$ and $f \in \L^p_{\sigma} (\Omega)$, the following two identities hold for all $t \geq 0$
\begin{equation*}
 \quad \e^{- t \cA_p} u = \big[A^{\frac{1}{2}}_{p'}\big]^*\Phi^{-1}  \e^{- t A_p} \Phi \big[ A_{p^{\prime}}^{- \frac{1}{2}} \big]^* u
 \quad\text{and}\quad 
 \Phi^{-1} \e^{- t A_p} f = \e^{- t \cA_p} \Phi^{-1} f.
\end{equation*}
 In particular, the weak Stokes semigroups are consistent on the $\W^{-1,p}_\sigma$-scale.
\item\label{it: MaxReg weak Stokes} $\mathcal{A}_p$ has maximal $\L^q$-regularity for $1 < q < \infty$.
\end{enumerate}
\end{proposition}

\section[Global strong solutions to the Navier--Stokes equations in planar Lipschitz domains]{Global strong solutions\texorpdfstring{\\}{} to the Navier--Stokes equations in planar Lipschitz domains}\label{Sec: Navier-Stokes}

\noindent This section is devoted to proving the regularity properties of Leray--Hopf weak solutions to the Navier--Stokes equations~\eqref{Eq: NSE} stated in Theorem~\ref{Thm: Navier-Stokes regularity}. Let
\begin{align*}
 u \in \L\H_{\infty} := \L^{\infty} (0 , \infty ; \L^2_{\sigma} (\Omega)) \cap \L^2 (0 , \infty ; \W^{1 , 2}_{0 , \sigma} (\Omega))
\end{align*}
be such a solution. Sufficient conditions for the existence of $u$ on $u_0$ and $f$ are that $u_0 \in \L^2_{\sigma} (\Omega)$ and that $f = f_0 + \IP_2 \divergence(F)$ with $f_0 \in \L^1 (0 , \infty ; \L^2_{\sigma} (\Omega))$ and $F \in \L^2(0 , \infty ; \L^2 (\Omega ; \IC^{2 \times 2}))$, see, e.g.,~\cite[Thm.~V.3.1.1]{Sohr}. These solutions further satisfy the \emph{energy inequality}
\begin{align*}
 E_\infty(u) \coloneqq &\frac{1}{2} \| u \|_{\L^{\infty} (0 , \infty ; \L^2 (\Omega; \IC^2))}^2 + \| \nabla u \|_{\L^2 (0 , \infty ; \L^2 (\Omega; \IC^{2\times 2}))}^2 \\
 &\quad\leq 2 \| u_0 \|_{\L^2_{\sigma} (\Omega)}^2 + 4 \| F \|_{\L^2 (0 , \infty ; \L^2 (\Omega; \IC^{2\times 2}))}^2 + 8 \| f_0 \|_{\L^1 (0 , \infty ; \L^2_{\sigma} (\Omega))}^2,
\end{align*}
see, e.g.,~\cite[Eq.~V.(3.1.13)]{Sohr}. The key to establishing the higher regularity result for Leray--Hopf weak solutions lies in the following two nonlinear estimates whose proofs can be found in~\cite[Lem.~V.1.2.1 and Rem.~V.1.2.2]{Sohr}.

\begin{lemma}
\label{Lem: Nonlinearity}
\begin{enumerate}[label=\normalfont{(\roman*)}]
 \item \label{Item: Lp spaces} For all $1 \leq s < 2$ there exists a constant $C > 0$ depending only on $s$ such that, for all $1 \leq p < 2$ satisfying
 \begin{align*}
  \frac{1}{p} + \frac{1}{s} = \frac{3}{2}
 \end{align*}
 and all
 \begin{align*}
  u \in \L^{\infty} (0 , \infty ; \L^2_{\sigma} (\Omega)) \cap \L^2 (0 , \infty ; \W^{1 , 2}_{0 , \sigma} (\Omega)),
 \end{align*}
 one has
 \begin{align*}
  \| (u \cdot \nabla) u \|_{\L^s (0 , \infty ; \L^p (\Omega; \IC^2))} \leq C E_\infty(u).
 \end{align*}
 \item For all $1 \leq s \leq \infty$, there exists a constant $C > 0$ depending only on $s$ such that, for all $1 \leq p < \infty$ satisfying
 \begin{align*}
  \frac{1}{p} + \frac{1}{s} = 1
 \end{align*}
 and all
 \begin{align*}
  u \in \L^{\infty} (0 , \infty ; \L^2_{\sigma} (\Omega)) \cap \L^2 (0 , \infty ; \W^{1 , 2}_{0 , \sigma} (\Omega)),
 \end{align*}
 one has
 \begin{align*}
  \| u \otimes u \|_{\L^s (0 , \infty ; \L^p (\Omega; \IC^{2\times 2}))} \leq C E_\infty(u).
 \end{align*}
\end{enumerate}
\end{lemma}

Relying on the maximal regularity results proved in Section~\ref{Sec: Resolvent estimates and maximal regularity}, we are now in the position to prove Theorem~\ref{Thm: Navier-Stokes regularity}.

\begin{proof}[Proof of Theorem~\ref{Thm: Navier-Stokes regularity}]
 \emph{Ad \ref{Item: Lp theory}}: 
 First of all, notice that~\cite[Thm.~IV.2.4.1]{Sohr} implies that, for almost every $t > 0$, one has that
 \begin{align*}
  \int_0^t \e^{- (t - \tau) A_2} A_2^{- \frac{1}{2}} \IP_2 \divergence(u (\tau) \otimes u(\tau)) \, \d \tau \in \dom(A_2^{\frac{1}{2}})
 \end{align*}
 and the representation formula
 \begin{align*}
  u(t) &= \e^{- t A_2} u_0 + \int_0^t \e^{- (t - s) A_2} f_0 (\tau) \, \d \tau 
  - A_2^{\frac{1}{2}} \int_0^t \e^{- (t - \tau) A_2} A_2^{- \frac{1}{2}} \IP_2 \divergence(u (\tau) \otimes u(\tau)) \, \d \tau,
 \end{align*}
 see~\cite[Thm.~V.1.3.1]{Sohr}. 
 Since the Stokes semigroups are consistent on the $\L^p_{\sigma}$-scale and since by assumption $u_0 \in \L^p_{\sigma} (\Omega)$ and $f_0 \in \L^s (0 , \infty ; \L^p_{\sigma} (\Omega))$, one can replace the $\L^2_{\sigma}$-semigroups by the $\L^p_{\sigma}$-semigroups. 
 Notice further that $\divergence(u (\tau) \otimes u(\tau)) = (u(\tau) \cdot \nabla) u(\tau)$ since $u$ is divergence-free. Furthermore, Lemma~\ref{Lem: Nonlinearity}~\ref{Item: Lp spaces} implies that $(u \cdot \nabla) u \in \L^s (0 , \infty ; \L^p (\Omega ; \IC^2))$. In particular, for almost every $0 < \tau < t$, one has that $(u(\tau) \cdot \nabla) u(\tau) \in \L^p (\Omega ; \IC^2)$ which implies that
 \begin{align*}
  \e^{- (t - \tau) A_2} A_2^{- \frac{1}{2}} \IP_2 \divergence\big(u (\tau) \otimes u(\tau)\big) = A_2^{- \frac{1}{2}} \e^{- (t - \tau) A_p} \IP_p (u (\tau) \cdot \nabla) u(\tau).
 \end{align*}
 Since $A_2^{- \frac{1 }{2}}$ is a bounded operator on $\L^2_{\sigma} (\Omega)$, one can pull this operator in front of the integral, yielding
 \begin{align*}
  u(t) &= \e^{- t A_p} u_0 + \int_0^t \e^{- (t - s) A_p} f_0 (\tau) \, \d \tau - \int_0^t \e^{- (t - \tau) A_p} \IP_p (u (\tau) \cdot \nabla) u(\tau) \, \d \tau.
 \end{align*}
 Consequently, $u$ is a mild solution to the \textit{linear} equation
 \begin{align*}
  \left\{ \begin{aligned}
   u^{\prime} (t) + A_p u(t) &= f_0 (t) - \IP_p (u(t) \cdot \nabla) u(t), && t > 0, \\
   u(0) &= u_0.
  \end{aligned} \right.
 \end{align*}
 Theorem~\ref{Thm: Maximal regularity} on the maximal regularity of the Stokes operator now implies that
 \begin{align*}
 u \in \W^{1 , s} (0 , \infty ; \L^p_{\sigma} (\Omega)) \cap \L^s (0 , \infty ; \dom(A_p)).
\end{align*}

\emph{Ad~\ref{Item: Distribution theory}}: 
Let $u$ be the Leray--Hopf weak solution corresponding to the initial value $u(0) = u_0$ and the right-hand side $f = \IP_2 \divergence(F)$.
In particular, $u$ can be represented via the representation formula
 \begin{equation}\label{eq: variation-of-constants rhs divergence}
  u(t) = \e^{- t A_2} u_0 
  + A_2^{\frac{1}{2}} \int_0^t \e^{- (t - \tau) A_2} A_2^{- \frac{1}{2}} \, \IP_2 \divergence\big(F(\tau) - u (\tau) \otimes u(\tau)\big) \, \d \tau,
 \end{equation}
 for almost every $t > 0$.
 As in Subsection~\ref{sec: weak Stokes}, let again 
 $\Phi \colon [\L^{p^{\prime}}_{\sigma}(\Omega)]^{*} \rightarrow \L^p_{\sigma}(\Omega)$
 denote  the  canonical  isomorphism between $[\L^{p^{\prime}}_{\sigma}(\Omega)]^{*}$ and $\L^p_{\sigma}(\Omega)$.
We will need the identity
\begin{equation}\label{eq: Characterization Pdiv}
    \Phi^{-1} A_2^{-\frac{1}{2}} \,\IP_2 \divergence(F) = \big[ A_2^{-\frac{1}{2}} \big]^* \,\IP_2 \divergence(F)
\end{equation}
on $\W^{-1,2}_\sigma(\Omega)$.
Indeed, the characterization of $A_2^{-\frac{1}{2}} \, \IP_2 \divergence$ by means of functionals as described in~\cite[Lem.~III.2.6.2]{Sohr} gives  that, for $F \in \L^2(\Omega; \IC^{2 \times 2})$ and $v \in \C_{0,\sigma}^\infty(\Omega)$,
\begin{align*}
\big\langle\, \Phi^{-1} A_2^{-\frac{1}{2}} \,\IP_2 \divergence(F) ,v \,\big\rangle_{\W^{-1,2}_{\sigma}, \W^{1,2}_{0,\sigma}} 
    &= \big\langle\, A_2^{-\frac{1}{2}} \,\IP_2 \divergence(F), v \,\big\rangle_{\L^{ 2}_\sigma , \L^{2}_\sigma} 
    = - \big\langle\, \IP_2 \divergence(F) , A_2^{-\frac{1}{2}} v  \,\big\rangle_{\W^{-1,2}_{\sigma}, \W^{1,2}_{0,\sigma}} \\
    &=\big\langle\, \big[ A_2^{-\frac{1}{2}} \big]^* \,\IP_2 \divergence(F) ,v \,\big\rangle_{\W^{-1,2}_{\sigma}, \W^{1,2}_{0,\sigma}} \,.
\end{align*}
If we now bring the representation of the semigroup from Proposition~\ref{Prop: Analyticity of weak Stokes}~\ref{it: consistency} together with identity~\eqref{eq: Characterization Pdiv}, we see that
\begin{equation}\label{Eq: similarity transform}
    \Phi^{-1} A_2^{\frac{1}{2}} \e^{- t A_2} A_2^{-\frac{1}{2}} \, \IP_2 \divergence(F ) 
    = 
    \e^{- t \cA_2} \, \IP_2 \divergence(F) 
\end{equation}
for all $F \in \L^{2}(\Omega; \IC^{2 \times 2})$ and $t > 0$.
Following the idea of the proof of~\cite[Thm.~3.3]{Choudhury_Hussein_Tolksdorf}, 
applying the embedding~$\Phi^{-1}$ to equation~\eqref{eq: variation-of-constants rhs divergence} and using~\eqref{Eq: similarity transform} and the consistency of the semigroups $\e^{-t \cA_p}$, see Proposition~\ref{Prop: Analyticity of weak Stokes}~\ref{it: consistency}, yields
 \begin{align*}
  \Phi^{-1} u(t) 
  &= \e^{- t \cA_p} \, \Phi^{-1} u_0 
  + \int_0^t \e^{- (t - \tau) \cA_p} \, \IP_p \divergence\big(F(\tau) - u (\tau) \otimes u(\tau)\big) \, \d \tau,
 \end{align*}
for almost every $t >0$.
 Consequently, $\Phi^{-1} u$ is a mild solution to the \textit{linear} equation
 \begin{align*}
  \left\{ \begin{aligned}
   \mathcal{U}^{\prime} (t) + \cA_p \,\mathcal{U}(t) &= \IP_p \divergence \big(F(t) - u(t) \otimes u(t) \big), && t > 0, \\
   \mathcal{U}(0) &= \Phi^{-1} u_0.
  \end{aligned} \right.
 \end{align*}
Proposition~\ref{Prop: Analyticity of weak Stokes}~\ref{it: MaxReg weak Stokes} on the maximal regularity of the weak Stokes operator now implies 
 \begin{align*}
 \Phi^{-1} u \in \W^{1 , s} (0 , \infty ; \W^{-1,p}_{\sigma} (\Omega)) \cap \L^s (0 , \infty ; \Phi^{-1} \W^{1,p}_{0,\sigma}(\Omega)).
\end{align*}
We arrive at the desired regularity result by translating this result into terms of $u$. 
\end{proof}

Combining Theorem~\ref{Thm: Navier-Stokes regularity} with embeddings of $\dom(A_p)$ into Bessel potential spaces leads to the following corollary.

\begin{corollary}
In the situation of Theorem~\ref{Thm: Navier-Stokes regularity}~\ref{Item: Lp theory}, the solution $u$ satisfies 
\begin{align*}
 u \in \W^{1 , s} (0 , \infty ; \L^p_{\sigma} (\Omega)) \cap \L^s (0 , \infty ; \H^{\alpha , p}_{0 , \sigma} (\Omega))
\end{align*}
for any $0 < \alpha < 1 + 1 / p$.
\end{corollary}

\begin{proof}
Notice that $\dom(A_p)$ embeds for $1 < p < 2$ continuously into $\H^{\alpha , p}_{0 , \sigma} (\Omega)$ by Remark~\ref{Rem: Domain embedding}. The rest follows by Theorem~\ref{Thm: Navier-Stokes regularity}~\ref{Item: Lp theory}.
\end{proof}

\appendix
\begin{appendix}

\section{The proofs of Lemma~\ref{Lem: Difference Helmholtz} and Theorem~\ref{Thm: Comparison with stationary case}}

\noindent This appendix provides the missing calculations for the proofs of Lemma~\ref{Lem: Difference Helmholtz} and Theorem~\ref{Thm: Comparison with stationary case}.
We build on the notations already established at the beginning of Subsection~\ref{Sec: Properties of the matrix of fundamental solutions to the Stokes resolvent problem}, and we collect expressions for the partial derivatives of the fundamental solutions to the Laplace equation and the scalar Helmholtz equation in $d= 2$.

For the fundamental solution $G(x; 0) = -\frac{1}{2\pi} \log(|x|)$ to the Laplace equation in two dimensions, the partial derivatives read:
\begingroup
\allowdisplaybreaks
\begin{align*}
  \partial_\gamma G(x; 0) &= -\frac{1}{2\pi} \, \frac{x_\gamma}{|x|^2}, \\[1.0em]
  \partial_\alpha \partial_\gamma G(x; 0) &= -\frac{1}{2\pi}\, \frac{\delta_{\alpha\gamma}}{|x|^2} + \frac{1}{\pi} \, \frac{x_\alpha x_\gamma}{|x|^4}, \\[1.0em]
  \partial_\beta \partial_\alpha \partial_\gamma G(x; 0)
  &= \frac{1}{\pi} \, \frac{\delta_{\beta \gamma} x_\alpha + \delta_{\alpha \gamma} x_\beta + \delta_{\alpha\beta} x_\gamma}{|x|^4} - \frac{4}{\pi}\, \frac{x_\alpha x_\beta x_\gamma}{|x|^6},
  \qquad \alpha, \beta, \gamma \in \{1, 2\}.
\end{align*}
\endgroup

Recall from Subsection~\ref{Sec: Properties of the matrix of fundamental solutions to the Stokes resolvent problem} that the fundamental solution for the scalar Helmholtz equation~\eqref{Eq: Scalar Helmholtz} is given via 
$
  G(x; \lambda) = \frac{\ii}{4}\, H_0^{(1)}(k|x|),
$
where $H_0^{(1)}(z)$ is the Hankel function of the first kind.
Denoting complex derivatives by $\frac{\d{}}{\d z}$, we calculate using the chain rule and the product rule:
\begingroup
\allowdisplaybreaks
\begin{align*}
  \partial_\gamma G(x; \lambda) &= \frac{\ii}{4}\, k\;  \frac{x_\gamma}{|x|}\, \frac{\d{}}{\d z} H_0^{(1)}(k |x|)\,, \\[1.0em]
  \partial_\alpha \partial_\gamma G(x; \lambda) 
  &=  \frac{\ii}{4}\,  k \,\bigg(\frac{\delta_{\alpha\gamma}}{|x|} - \frac{x_\alpha x_\gamma}{|x|^3} \bigg)\, \frac{\d{}}{\d z} H_0^{(1)}(k |x|) + \frac{\ii}{4} \, k^2 \;  \frac{x_\alpha x_\gamma}{|x|^2} \, \frac{\d{}^2}{\d z^2} H_0^{(1)}(k |x|)\,, \\[1.0em]
  \partial_\beta \partial_\alpha \partial_\gamma G(x; \lambda)
  &= \frac{\ii}{4} \, k^3 \; \frac{x_\alpha x_\beta x_\gamma}{|x|^3} \, \frac{\d{}^3}{\d z^3} H_0^{(1)}(k |x|)\\*[0.5em]
  &\quad + \frac{\ii}{4} \, k^2 \,\bigg(\frac{\delta_{\beta\gamma} x_\alpha + \delta_{\alpha \gamma} x_\beta + \delta_{\alpha \beta} x_\gamma}{|x|^2} - 3\, \frac{x_\alpha x_\beta x_\gamma}{|x|^4} \bigg)\, \frac{\d{}^2}{\d z^2} H_0^{(1)}(k |x|) \\*[0.5em]
  &\quad+ \frac{\ii}{4} \, k \, \bigg( 3\, \frac{x_\alpha x_\beta x_\gamma}{|x|^5} - \frac{\delta_{\beta\gamma} x_\alpha + \delta_{\alpha\gamma} x_\beta + \delta_{\alpha\beta} x_\gamma}{|x|^3} \bigg)\, \frac{\d{}}{\d z} H_0^{(1)}(k |x|) \,,
  \quad \alpha, \beta, \gamma \in \{1, 2\}.
\end{align*}
\endgroup

We will now present the derivatives of the Hankel function $H_0^{(1)}(z)$ that appeared in the previous calculations by calculating the complex derivatives of the series expansion of $H_0^{(1)}(z)$ which may be found in Lebedev~\cite[Sect.\@~5.6]{Lebedev}. 
The expansions below are a consequence of the series expansions of the \emph{Bessel functions of the first kind} $J_0$ and the \emph{Bessel functions of the second kind} $Y_0$.

Let $\psi$ denote the \emph{digamma function}, i.e., the logarithmic derivative of the \emph{gamma} function. 
The previous expansions read:
\begingroup
\allowdisplaybreaks
\begin{align*}
  \frac{\pi}{2 \ii} \, H_0^{(1)}(z)
  &= \frac{\pi}{2\ii} \, J_0(z) + \frac{\pi}{2}\, Y_0(z) \\*
  &= \sum_{\ell = 0}^\infty \, \frac{(-1)^\ell}{(\ell!)^2 \, 4^\ell} \,  z^{2\ell} \Big( -\frac{\ii \pi}{2} - \log(2) - \psi(\ell + 1) \Big) 
  +  \sum_{\ell = 0}^\infty \, \frac{(-1)^\ell}{(\ell!)^2 \, 4^\ell} \,  z^{2\ell} \log(z) \\*
  &= \sum_{\ell = 0}^\infty \, a_\ell \,  z^{2\ell} \, C_\ell 
   + \sum_{\ell = 0}^\infty \, a_\ell  \,z^{2\ell} \log(z)\,, \\[1.0em]
  \frac{\pi}{2\ii} \, \frac{\d{}}{\d z} H_0^{(1)}(z)
  &= \sum_{\ell = 1}^\infty \, a_\ell \,  (2\ell) \, z^{2\ell - 1} \, C_\ell 
  + \sum_{\ell = 1}^\infty \, a_\ell \, (2\ell)  \, z^{2\ell - 1} \log(z) \, 
  + \sum_{\ell = 0}^\infty \, a_\ell \, z^{2\ell - 1} \\*
  &= \sum_{\ell = 1}^\infty \, b_\ell \, z^{2\ell - 1} \, C_\ell 
  + \sum_{\ell = 1}^\infty \, b_\ell \, z^{2\ell - 1} \log(z) \, 
  +  \sum_{\ell = 0}^\infty \, a_\ell \, z^{2\ell - 1}, \\[1.0em]
  \frac{\pi}{2\ii} \, \frac{\d{}^2}{\d z^2} H_0^{(1)}(z)
  &= \sum_{\ell = 1}^\infty \, b_\ell \, (2\ell - 1) \, z^{2\ell - 2} \, C_\ell 
  + \sum_{\ell = 1}^\infty \, b_\ell \, (2\ell - 1)\, z^{2\ell - 2} \log(z) 
  + \sum_{\ell = 1}^\infty \, b_\ell \, z^{2\ell - 2}\\*
  &\quad + \sum_{\ell = 0}^\infty \, a_\ell \, (2\ell - 1) \, z^{2\ell - 2} \\*
  &=   \sum_{\ell = 1}^\infty c_\ell \, z^{2\ell - 2} \, C_\ell 
  +  \!\sum_{\ell = 1}^\infty c_\ell \, z^{2\ell - 2} \log(z) 
  +  \!\sum_{\ell = 1}^\infty b_\ell \, z^{2\ell - 2} 
  +  \!\sum_{\ell = 0}^\infty a_\ell \, (2\ell - 1) \, z^{2\ell - 2}, \\[1.0em]
  \frac{\pi}{2\ii} \, \frac{\d{}^3}{\d z^3} H_0^{(1)}(z)
  &= \sum_{\ell = 2}^\infty c_\ell \, (2\ell - 2) \, z^{2\ell - 3} \, C_\ell 
  +  \sum_{\ell = 2}^\infty \, c_\ell \, (2\ell - 2) \, z^{2\ell - 3} \log(z) 
  +  \sum_{\ell = 1}^\infty \, c_\ell \, z^{2\ell - 3} \\*
 &\quad  
  + \sum_{\ell = 2}^\infty \, b_\ell \, (2\ell - 2) \, z^{2\ell - 3} 
  + \sum_{\ell = 0}^\infty \, a_\ell \, (2\ell - 1) \, (2\ell - 2) \, z^{2\ell - 3} \\*
  &= \sum_{\ell = 2}^\infty \, d_\ell \, z^{2\ell - 3} \, C_\ell 
  + \sum_{\ell = 2}^\infty \, d_\ell \, z^{2\ell - 3} \log(z) 
  + \sum_{\ell = 1}^\infty \, c_\ell \, z^{2\ell - 3} \\
 &\quad  + \sum_{\ell = 2}^\infty \, b_\ell \, (2\ell - 2) \, z^{2\ell - 3} 
  + \sum_{\ell = 0}^\infty \, a_\ell \, (2\ell - 1)\, (2\ell - 2) \, z^{2\ell - 3} ,
\end{align*}
where we introduced the following coefficients for better readability:
\begin{align}
  \begin{alignedat}{1}
  C_\ell &\coloneqq -\frac{\ii \pi}{2} - \log(2) - \psi(\ell + 1), \\
  a_\ell &\coloneqq \frac{(-1)^\ell}{(\ell!)^2 \, 4^\ell}, \quad
  b_\ell \coloneqq a_\ell \cdot  2\ell, \quad
  c_\ell \coloneqq b_\ell \cdot (2\ell - 1), \quad\text{and}\quad
  d_\ell \coloneqq c_\ell \cdot (2\ell - 2).
  \end{alignedat}\label{defnConst}\tag{D1}
\end{align}
\endgroup

\subsection{Proof of Lemma~\ref{Lem: Difference Helmholtz} for \texorpdfstring{$d = 2$}{d = 2}}\label{sec:A1}

For $\lambda \in \S_\theta$, $\theta \in (0, \pi/2)$, and $|\lambda| |x|^2 \leq (1/2)$, we show that the estimate
\begin{align*}
\sup_{\alpha, \beta, \gamma \in \{1,2\}}
  \Big|\, \partial_\beta \partial_\alpha \partial_\gamma G(x; \lambda) - \partial_\beta \partial_\alpha \partial_\gamma G(x; 0)  \, \Big|
  \leq C\, |\lambda| |x|^{-1}
\end{align*}
is valid with a constant $C > 0$ that only depends on $\theta$.
Fix $\alpha$, $\beta$, and $\gamma$.
We proceed in two steps. In the first step, we filter out all summands in the series expansion of  
$\partial_\beta \partial_\alpha \partial_\gamma G(x; \lambda)$ for which an individual estimate does not yield the desired bound in Lemma~\ref{Lem: Difference Helmholtz}.
We will call the terms whose absolute value cannot be bounded individually \emph{problematic}.
In the second step, we show that taking the sum over all problematic terms and subtracting $\partial_\beta \partial_\alpha \partial_\gamma G(x; 0)$ gives 0.

In order to make the next calculations better to digest, we decompose the third derivative of $G(\,\cdot\, ; \lambda)$ as follows:
\begin{align*}
  \partial_\beta \partial_\alpha \partial_\gamma G(x; \lambda)
  \eqqcolon A_3 + A_2 + A_1,
\end{align*}
where each $A_i$, $i = 1,\dots,3$, corresponds to the term involving the $i$th complex derivative of $H_0^{(1)}$.
Let us start with $A_3$. 
We have
\begin{align*}
  A_3 = - \frac{1}{2 \pi } k^3 \, \frac{x_\alpha x_\beta x_\gamma}{|x|^3}  \,
  &\cdot\,\bigg\{
     \; \sum_{\ell = 2}^\infty \; d_\ell \, (k|x|)^{2\ell - 3} \, C_\ell +  \sum_{\ell = 2}^\infty \; d_\ell \, (k|x|)^{2\ell - 3} \log(k|x|)  \\
  &\quad\,  +  \sum_{\ell = 1}^\infty \; c_\ell \, (k|x|)^{2\ell - 3} +  \sum_{\ell = 2}^\infty \; b_\ell\, (2\ell - 2) \, (k|x|)^{2\ell - 3} \\
  &\quad\, +  \sum_{\ell = 0}^\infty \; a_\ell \, (2\ell - 1)\, (2\ell - 2) \, (k|x|)^{2\ell - 3} \;
  \bigg\}.
\end{align*}
First, note that
\begin{align*}
  \Big|\, \frac{1}{2\pi} k^3\, \frac{x_\alpha x_\beta x_\gamma}{|x|^3} \,\cdot\; c_1\, (k|x|)^{2 \cdot 1 -3} \,\Big| 
  \leq C \, |\lambda| |x|^{-1},
\end{align*}
with a constant $C > 0$.
This shows that the first term of the third sum in $A_3$ is not problematic.
For the rest of $A_3$, we use the fact $|\lambda| |x|^2 \leq (1/2)$ to trade one $|k| = \sqrt{|\lambda|}$ in the prefactor for a constant times $|x|^{-1}$ and show that the prefactor behaves as
\begin{align*}
  \Big|\, k^3 \, \frac{x_\alpha x_\beta x_\gamma}{|x|^3}\, \Big| 
  \leq C\, |\lambda| |x|^{-1}, 
\end{align*}
with a constant $C > 0$.
Therefore, the only problematic term in $A_3$ is the first element of the last sum
\begin{align}
  \label{eq:P1}
  \tag{P1}- \frac{1}{2 \pi } k^3\, \frac{x_\alpha x_\beta x_\gamma}{|x|^3} \, \cdot\;  a_0 \, (2 \cdot 0 - 1)\, (2 \cdot 0 - 2) \, (k|x|)^{2 \cdot 0 - 3} .
\end{align}
For $A_2$, we calculate
\begin{align*}
  A_2 &= - \frac{1}{2\pi} k^2\, \bigg(\, \frac{\delta_{\beta\gamma} x_\alpha + \delta_{\alpha \gamma} x_\beta + \delta_{\alpha \beta} x_\gamma}{|x|^2} - 3\, \frac{x_\alpha x_\beta x_\gamma}{|x|^4} \, \bigg) \\ 
  &\qquad \cdot 
  \bigg\{
    \; \sum_{\ell = 1}^\infty \; c_\ell \, (k|x|)^{2\ell - 2} \, C_\ell 
  + \; \sum_{\ell = 1}^\infty \; c_\ell \, (k|x|)^{2\ell - 2} \log(k|x|) \\
  &\qquad\quad\,+ \; \sum_{\ell = 1}^\infty \; b_\ell \, (k|x|)^{2\ell - 2} 
  + \; \sum_{\ell = 0}^\infty \; a_\ell\, (2\ell - 1) \, (k|x|)^{2\ell - 2} 
  \bigg\}.
\end{align*}
As the prefactor already behaves like $|\lambda| |x|^{-1}$, we identify the following two terms as being problematic:
\begin{align}
  \label{eq:P2}
  \tag{P2}
  \begin{alignedat}{1}
  & - \frac{1}{2\pi} k^2\, \bigg(\, \frac{\delta_{\beta\gamma} x_\alpha + \delta_{\alpha \gamma} x_\beta + \delta_{\alpha \beta} x_\gamma}{|x|^2} - 3\, \frac{x_\alpha x_\beta x_\gamma}{|x|^4} \, \bigg) 
    \, \cdot\;  c_1 \, (k|x|)^{2 \cdot 1 - 2} \log(k|x|) \;\;\;\text{and}\\[0.5em]
   &- \frac{1}{2\pi} k^2 \,\bigg(\, \frac{\delta_{\beta\gamma} x_\alpha + \delta_{\alpha \gamma} x_\beta + \delta_{\alpha \beta} x_\gamma}{|x|^2} - 3\, \frac{x_\alpha x_\beta x_\gamma}{|x|^4} \, \bigg) 
  \, \cdot \; a_0\, (2\cdot 0 - 1) \, (k|x|)^{2 \cdot 0 - 2} .
  \end{alignedat}
\end{align}
For the last component, we have the following identity:
\begin{align*}
  A_1 = 
  & - \frac{1}{2\pi} k\, \bigg( 3\, \frac{x_\alpha x_\beta x_\gamma}{|x|^5} - \frac{\delta_{\beta\gamma} x_\alpha + \delta_{\alpha\gamma} x_\beta + \delta_{\alpha\beta} x_\gamma}{|x|^3} \bigg)  \\
  &\,\cdot\, \bigg\{ 
  \; \sum_{\ell = 1}^\infty \; b_\ell \, (k|x|)^{2\ell - 1} \, C_\ell 
  + \sum_{\ell = 1}^\infty \; b_\ell \, (k|x|)^{2\ell - 1} \log(k|x|) 
  + \sum_{\ell = 0}^\infty \; a_\ell \, (k|x|)^{2\ell - 1} 
  \bigg\}.
\end{align*}
In this case, the prefactor behaves like $\sqrt{|\lambda|} |x|^{-2}$.
Therefore, problematic terms only arise in the last two sums
\begin{align}
  \label{eq:P3}
  \tag{P3}
  \begin{alignedat}{1}
    &- \frac{1}{2\pi} k\, \bigg( \, 3\, \frac{x_\alpha x_\beta x_\gamma}{|x|^5} - \frac{\delta_{\beta\gamma} x_\alpha + \delta_{\alpha\gamma} x_\beta + \delta_{\alpha\beta} x_\gamma}{|x|^3} \, \bigg)
    \, \cdot \; b_1 \, (k|x|)^{2 \cdot 1 - 1} \log(k|x|)  \quad\text{and}\\[0.5em]
    &- \frac{1}{2\pi} k\, \bigg( \, 3\, \frac{x_\alpha x_\beta x_\gamma}{|x|^5} - \frac{\delta_{\beta\gamma} x_\alpha + \delta_{\alpha\gamma} x_\beta + \delta_{\alpha\beta} x_\gamma}{|x|^3} \, \bigg)
  \, \cdot \; a_0 \, (k|x|)^{2 \cdot 0 - 1} .
  \end{alignedat}
\end{align}
Note that Definition~\eqref{defnConst} gives $a_0 = 1$ and $c_1 = b_1 = - 1/2$.
Therefore, we can already see that the logarithmic terms in the sum $\eqref{eq:P2} + \eqref{eq:P3}$ cancel.

Finally, we observe that, if we take the sum over the remaining problematic terms in \eqref{eq:P1}, \eqref{eq:P2}, and \eqref{eq:P3} and subtract $\partial_\beta \partial_\alpha \partial_\gamma G(x; 0)$, the result is $0$. 
The former fact is easily seen by grouping the terms having the same power of $|x|$:
\begin{samepage}
\begin{align*}
  &\hspace{-2cm}\mathrm{(P1)} + \mathrm{(P2)} + \mathrm{(P3)} - \partial_\beta \partial_\alpha \partial_\gamma G(x; 0) \\[0.5em]
  = &- \frac{1}{\pi}\, \frac{x_\alpha x_\beta x_\gamma}{|x|^6} \\[0.5em]
    &+ \frac{1}{2\pi} \, \bigg(\,\frac{\delta_{\beta\gamma} x_\alpha + \delta_{\alpha \gamma} x_\beta + \delta_{\alpha \beta} x_\gamma}{|x|^4} - 3\, \frac{x_\alpha x_\beta x_\gamma}{|x|^6} \, \bigg) \\[0.5em]
    &- \frac{1}{2\pi} \, \bigg( \, 3\, \frac{x_\alpha x_\beta x_\gamma}{|x|^6} - \frac{\delta_{\beta\gamma} x_\alpha + \delta_{\alpha\gamma} x_\beta + \delta_{\alpha\beta} x_\gamma}{|x|^4} \, \bigg)\\[0.5em]
   &- \frac{1}{\pi}\, \frac{\delta_{\beta \gamma} x_\alpha + \delta_{\alpha \gamma} x_\beta + \delta_{\alpha\beta} x_\gamma}{|x|^4} 
     + \frac{4}{\pi}\, \frac{x_\alpha x_\beta x_\gamma}{|x|^6} 
     \quad=\quad 0 \,. 
\end{align*}
This completes the proof of Lemma~\ref{Lem: Difference Helmholtz}. \hfill$\qed$
\end{samepage}

\subsection{Proof of Theorem~\ref{Thm: Comparison with stationary case}}\label{sec:A2}

In this section, we show that, for all $\lambda \in \S_\theta$, $\theta \in (0,\pi/2)$, and $x \in \IR^2\setminus \{0\}$ subject to the condition $|\lambda| |x|^2 \leq (1/2)$, we have
\begin{align*}
  \Big|\, \nabla_x \big\{ \Gamma(x; \lambda) - \Gamma(x; 0) \big\} \, \Big| \leq C\, |\lambda| |x| \, \big|\log(|\lambda| |x|^2)\,\big|\,,
\end{align*}
where $C > 0$ depends only on $\theta$.
The means and the strategy to prove this estimate are similar to the procedure in Appendix~\ref{sec:A1}.
In addition to the derivatives that were calculated at the beginning of this appendix, we will furthermore need the first partial derivatives for the matrix of fundamental solutions of the Stokes problem:
\begin{align*}
    \Gamma_{\alpha\beta}(x; 0) 
    &= \frac{1}{4\pi} \bigg\{ \,- \delta_{\alpha\beta} \log(|x|) + \frac{x_\alpha x_\beta}{|x|^2} \,\bigg\} \,,
    \\[0.5em]
  \partial_\gamma \Gamma_{\alpha\beta} (x; 0) 
  &= \frac{1}{4 \pi}\, \bigg(\, \frac{\delta_{\alpha\gamma} x_\beta + \delta_{\beta\gamma} x_\alpha  - \delta_{\alpha\beta} x_\gamma}{|x|^2}  \, \bigg) 
  - \frac{1}{2\pi}\, \frac{x_\alpha x_\beta x_\gamma}{|x|^4},\quad \alpha, \beta, \gamma \in \{1,2\}.
\end{align*}
Now, consider the difference
\begin{align*}
  &\partial_\gamma \Gamma_{\alpha\beta}(x; \lambda) - \partial_\gamma \Gamma_{\alpha\beta}(x; 0) \\
  &\qquad=\partial_\gamma G(x; \lambda) \,\delta_{\alpha\beta}
  + \frac{1}{k^2} \, \partial_\beta \partial_\alpha \partial_\gamma \Big\{ G(x; \lambda) - G(x; 0) \Big\} 
  - \partial_\gamma \Gamma_{\alpha\beta}(x; 0) =: B_1 + B_2 + B_3 \,,
\end{align*}
where we introduced the variables $B_i$, $i = 1,\dots,3$, for the sake of readability.
As in the previous section, we will, in a first step, study the terms $B_1$ and $B_2$ independently and filter out the \emph{problematic} terms, i.e., those terms that do not individually meet the bound stated in Theorem~\ref{Thm: Comparison with stationary case}. 
In a second step, we will compare the found problematic terms with the term $B_3$ and show that, in the final sum, all terms add to zero.

Starting with $B_1$, we have
\begin{align*}
  B_1 = -\frac{1}{2\pi} k\;  \frac{\delta_{\alpha\beta} x_\gamma}{|x|}  \;
  \cdot \,\bigg\{ \;
  \sum_{\ell = 1}^\infty \; b_\ell \, (k|x|)^{2\ell - 1} \, C_\ell 
  &+ \sum_{\ell = 1}^\infty \; b_\ell \, (k|x|)^{2\ell - 1} \log(k|x|) \; \\
  &+  \sum_{\ell = 0}^\infty \; a_\ell \, (k|x|)^{2\ell - 1} \; \bigg\}.
\end{align*}
In this expression, we only detect one problematic term, namely the term corresponding to $\ell = 0$ in the last sum
\begin{align}
  \label{eq:Q1}
  \tag{Q1}
  -\frac{1}{2\pi} k\;  \frac{\delta_{\alpha\beta} x_\gamma}{|x|}\, \cdot \; a_0 \, (k|x|)^{2 \cdot 0 - 1} .
\end{align}

Recall that, in Section~\ref{sec:A1}, we showed that $\eqref{eq:P1} +\eqref{eq:P2} +  \eqref{eq:P3} - \partial_\beta\partial_\alpha\partial_\gamma G(x; 0) = 0$.  
For the expression $B_2$, this leads us to the decomposition
\begin{align*}
  B_2 
  &= k^{-2} \, \Big( \, A_3 + A_2 + A_1 - \partial_\beta\partial_\alpha\partial_\gamma G(x; 0)\, \Big)  \\
  &= k^{-2} \, \Big( \, \big\{ A_3 - \eqref{eq:P3} \big\} 
                       +\big\{ A_2 - \eqref{eq:P2} \big\}
                       +\big\{ A_1 - \eqref{eq:P1} \big\} \, \Big) \eqqcolon A_3' + A_2' + A_1'\,,
\end{align*}
with the variables $A_i$, $i = 1, \dots, 3$, which were also introduced in the previous section. 
In the following, we list the problematic terms surging from the terms in each $A_i'$. 

For $A_3'$, we see that the following term does not meet the desired behavior:
\begin{align}
  \label{eq:Q2}
  \tag{Q2}
  \begin{alignedat}{1}
  &- \frac{1}{2 \pi } k \, \frac{x_\alpha x_\beta x_\gamma}{|x|^3} \; \cdot \; c_1 \, (k|x|)^{2 \cdot 1 - 3}. 
  \end{alignedat}
\end{align}
For $A_2'$, we see that, compared to $A_2$, every summand is problematic which after cancellation leads to:
\begin{align}
  \label{eq:Q3}
  \tag{Q3}
  \begin{alignedat}{1}
  &- \frac{1}{2\pi} \, \bigg(\,\frac{\delta_{\beta\gamma} x_\alpha + \delta_{\alpha \gamma} x_\beta + \delta_{\alpha \beta} x_\gamma}{|x|^2} - 3\, \frac{x_\alpha x_\beta x_\gamma}{|x|^4} \, \bigg) 
    \,\cdot \; c_1 \, (k|x|)^{2\cdot 1 - 2} \, C_1 \\[1.0em]
  &- \frac{1}{2\pi} \, \bigg(\, \frac{\delta_{\beta\gamma} x_\alpha + \delta_{\alpha \gamma} x_\beta + \delta_{\alpha \beta} x_\gamma}{|x|^2} - 3\, \frac{x_\alpha x_\beta x_\gamma}{|x|^4} \, \bigg) 
    \,\cdot \; b_1 \, (k|x|)^{2 \cdot 1 - 2}  \\[1.0em]
  &- \frac{1}{2\pi} \, \bigg(\,\frac{\delta_{\beta\gamma} x_\alpha + \delta_{\alpha \gamma} x_\beta + \delta_{\alpha \beta} x_\gamma}{|x|^2} - 3\, \frac{x_\alpha x_\beta x_\gamma}{|x|^4} \, \bigg) 
  \,\cdot \; a_1 (2 \cdot 1 - 1) \, (k|x|)^{2 \cdot 1 - 2} .
  \end{alignedat}
\end{align}
The same holds for the expression $A_1'$:
\begin{align}
  \label{eq:Q4}
  \tag{Q4}
  \begin{alignedat}{1}
    &- \frac{1}{2\pi} \, \frac{1}{k} \, \bigg(\, 3\, \frac{x_\alpha x_\beta x_\gamma}{|x|^5} - \frac{\delta_{\beta\gamma} x_\alpha + \delta_{\alpha\gamma} x_\beta + \delta_{\alpha\beta} x_\gamma}{|x|^3} \, \bigg)  
    \,\cdot \; b_1 \, (k|x|)^{2 \cdot 1 - 1} \, C_1  \\[1.0em]
    &- \frac{1}{2\pi} \, \frac{1}{k} \, \bigg( \, 3\, \frac{x_\alpha x_\beta x_\gamma}{|x|^5} - \frac{\delta_{\beta\gamma} x_\alpha + \delta_{\alpha\gamma} x_\beta + \delta_{\alpha\beta} x_\gamma}{|x|^3} \, \bigg)  
  \,\cdot \; a_1 \, (k|x|)^{2 \cdot 1 - 1} .
  \end{alignedat}
\end{align}

Now it is time to sum all problematic terms, expand the variables and add the term $B_3$. 
With
$a_0 = 1$, $a_1 = -1/4$, and $c_1 = b_1 =  -1/2$,
we see that, within the sum \eqref{eq:Q3} + \eqref{eq:Q4}, a lot of terms cancel. 
This leaves us with:
\begin{align*}
  \label{eq:Q33}
  \tag{Q3${}^\prime$}
  &- \frac{1}{2\pi} \, \bigg(\, \frac{\delta_{\beta\gamma} x_\alpha + \delta_{\alpha \gamma} x_\beta + \delta_{\alpha \beta} x_\gamma}{|x|^2} - 3\, \frac{x_\alpha x_\beta x_\gamma}{|x|^4} \, \bigg) 
  \, \cdot b_1 \, (k|x|)^{2\cdot 1 - 2}.
\end{align*}
Finally, let us consider the remaining sum $\eqref{eq:Q1} + \eqref{eq:Q2} + \eqref{eq:Q33} + B_3$:
\begin{align*}
  &-\frac{1}{2\pi} \,  \frac{\delta_{\alpha\beta} x_\gamma}{|x|^2} 
  + \frac{1}{4 \pi }\, \frac{x_\alpha x_\beta x_\gamma}{|x|^4} 
  + \frac{1}{4\pi} \, \bigg(\,\frac{\delta_{\beta\gamma} x_\alpha + \delta_{\alpha \gamma} x_\beta + \delta_{\alpha \beta} x_\gamma}{|x|^2} - 3\, \frac{x_\alpha x_\beta x_\gamma}{|x|^4} \, \bigg) 
   \\[0.5em]
 &\quad - \frac{1}{4 \pi} \, \bigg(\, \frac{\delta_{\alpha\gamma} x_\beta + \delta_{\gamma\beta} x_\alpha  - \delta_{\alpha\beta} x_\gamma}{|x|^2} \, \bigg) 
  + \frac{1}{2\pi}\,  \frac{x_\alpha x_\beta x_\gamma}{|x|^4}
  \quad = \quad 0\, .
\end{align*}
As all problematic terms add up to zero, we have proved the initial claim.\hfill$\qed$
\end{appendix}

\end{document}